\documentclass{article}
\usepackage{amsmath}
\usepackage{amsfonts}
\usepackage{amssymb}
\usepackage{dsfont}

\usepackage{soul}
\usepackage{euscript}
\usepackage{xcolor}
\usepackage[colors]{optsys}

\def\mathscr{\EuScript}

\newcommand{\paramu}{\nu}
\newcommand{\ti}{t_{\mathrm{i}}}
\newcommand{\tf}{t_{\mathrm{f}}}

\newcommand{\SoN}{\xi}    % State of Nature
\newcommand{\SSoN}{\Xi}   % Space of State of Nature

\newcommand{\pif}{\pi_{\mathrm{f}}}

\newcommand{\tp}{\mathbf{T}_{\mathrm{p}}}
\newcommand{\td}{\mathbf{T}_{\mathrm{d}}}
\newcommand{\tpx}{t_{\mathrm{p}}}
\newcommand{\Tp}{\mathbf{T}_{\mathrm{p}}}
\newcommand{\tdx}{t_{\mathrm{d}}}
\newcommand{\Td}{\mathbf{T}_{\mathrm{d}}}
\newcommand{\tfx}{t_{\mathrm{f}}^{\SoN}}
\newcommand{\tpxk}{t_{\mathrm{p}}^{\SoN^{k}}}
\newcommand{\tdxk}{t_{\mathrm{d}}^{\SoN^{k}}}

\newcommand{\xx}{x^{\SoN}}

\newcommand{\stratv}{\mathbf{V}}

\newcommand{\fctI}{\mathrm{I}}

\newcommand{\Xtf}{\Psi^{(u,\va{V},\Tp,\Td)}_{\ti,\tf}(x_i)}

\usepackage{a4wide}
\usepackage{fullpage}
\usepackage{authblk}

\usepackage{ifpdf}
\ifpdf
  \usepackage[pdftex]{graphicx}
  \DeclareGraphicsExtensions{.pdf}
  \usepackage[pdftex]{hyperref}
\else
  \usepackage[dvips]{graphicx}
  \DeclareGraphicsExtensions{.eps}
  \usepackage[dvips]{hyperref}
\fi
\graphicspath{{./},{./Figures-new/}}

\newcommand{\nb}[3]{
  {\colorbox{#2}{\bfseries\sffamily\tiny\textcolor{white}{#1}}}
  {\textcolor{#2}{$\blacktriangleright${#3}$\blacktriangleleft$}}}

\newcommand{\jpc}[1]{\nb{jpc}{orange}{#1}}

\title{Low-thrust Interplanetary Trajectories with
  Missed Thrust Events: a Numerical Approach}

\author[1]{Pierre Carpentier}
\author[2]{Jean-Philippe Chancelier}
\author[2]{Guy Cohen}
\author[3]{Thierry Dargent}
\author[4]{Richard Epenoy}

\affil[1]{UMA, ENSTA Paris, IP Paris, France}
\affil[2]{CERMICS, \'Ecole nationale des ponts et chauss\'{e}es, IP Paris, France}
\affil[3]{Thales Alenia Space, France}
\affil[4]{Centre National d'\'Etudes Spatiales, France}

\date{\today}

\begin{document}

\maketitle

\begin{abstract}
The problem under consideration is to drive a spatial vehicle to a
target at a given final time while minimizing fuel consumption.
This is a classical optimal control problem in a deterministic
setting. However temporary stochastic failures of the engine may
prevent reaching the target after the engine usage is recovered.
Therefore, a stochastic optimal control problem is formulated under
the constraint of ensuring a minimal probability of hitting the target.
This problem is modeled, improved and finally solved by dualizing the
probability constraint and using an Arrow-Hurwicz stochastic algorithm.
Numerical results concerning an interplanetary mission are presented.
\end{abstract}

\section{Introduction and motivation\label{sect:introduction}}

The study described in this paper is the result of a scientific collaboration
conducted between Cermics, UMA and both the French government agency
CNES\footnote{\href{https://cnes.fr}{https://cnes.fr}} and Thales Alenia
Space\footnote{\href{https://www.thalesaleniaspace.com}{https://www.thalesaleniaspace.com}}.
The collaboration took place from December 2007 to December 2010.

\subsection{Problem statement\label{ssect:intro-problem}}

The primary objective of this work is to address the planning of space
rendez-vous missions, which can be formulated as optimal control problems with
terminal equality constraints. The cost function typically reflects goals such
as minimizing final time and/or fuel consumption.
However, the optimal control solutions derived from such formulations are
generally not ``robust''. That is, if a temporary engine failure occurs
--- causing a deviation from the ideal trajectory --- it may become very
difficult or even impossible to satisfy the terminal constraints.
To address this limitation, we propose an alternative approach in which the
final equality constraints are replaced by a \emph{constraint in probability},
based on a stochastic model of engine failure occurrence and duration.
The objective is to compute a reference trajectory that ensures a successful
rendez-vous with a specified probability, despite potential engine breakdowns.
Naturally, this reference trajectory entails a certain loss of performance,
as measured by the cost function, when compared to the ideal optimal solution.
The aim is therefore to make the trade-off between performance and safety more
explicit and quantifiable.

In more detail, we consider an optimal control problem involving a spacecraft
governed by differential equations. The control variable \(u\) represents the
action of the propulsion system, and the objective is to minimize fuel
consumption while reaching a target at a \emph{given} final time \(\tf\).
This type of problem can be addressed using Pontryagin’s Minimum Principle,
which yields an optimal control profile \(u(\cdot)\).
However, if a temporary engine failure occurs during the mission, the planned
control is forced to zero for the duration of the failure. Once the engine
recovers, the control must be recalculated based on the actual position at the
time of recovery in an attempt to still reach the target despite the
deviation. Yet, this may not always be feasible.
More specifically, given a probabilistic model for the breakdown time~\(\tp\)
and its duration~\(\td\), one can compute the probability of successfully
reaching the target after the engine recovers. Unfortunately, this probability
tends to be low if the spacecraft has strictly followed the nominal optimal
(deterministic) trajectory up to the point of failure.
As a result, to improve the chances of completing the rendez-vous successfully,
the trajectory must be planned from the outset to account for the possibility
of engine breakdowns.

\subsection{Literature review\label{ssect:intro-literature}}

%\rouge{Citations
%\begin{itemize}
%\item sur contraintes en probabilit\'{e}~:
%      Prekopa~\cite{Prekopa_Kluwer_1995},
%      Henrion~\cite{Henrion_JOTA_2002} - \cite{Henrion_MP_2004} -
%      \cite{Henrion_COA_2008},
%\item sur l'algorithme de Arrow-Hurwicz~:
%      Arrow~\cite{Arrow_SUP_1958},
%      Culioli et Cohen~\cite{Culioli_SIAM_1990} - \cite{Culioli_CRAS_1995} -
%      \cite{Culioli_RI_1994},
%\item sur le spatial~:
%      livre de Conway~\cite{Conway_CUP_2010} et articles \\
%      \cite{Chai_PAS_2019} - \cite{Ciancarelli_arXiv_2024} -
%      \cite{Hofmann_JGCD_2021} - \cite{Lantoine_PhD_2010} -
%      \cite{Morante_Aerospace_2021} - \cite{Morelli_arXiv_2023} -
%      \cite{Morelli_IEEE_2021} - \cite{Ross_JGCD_2007} -
%      \cite{Shirazi_PAS_2018} - \cite{Takubo_arXiv_2024} -
%      \cite{Topputo_AAA_2014}.
%\end{itemize}
%R\'{e}f\'{e}rences Richard~:
%  \cite{Sidhoum_AAS_2024} -
%  \cite{Shimane_ASCEND_2021} - \cite{Zavoli_JGCD_2021} -
%  \cite{Greco_JGCD_2022} - \cite{Oguri_AAS_2022}
%  et plus particuli\`{e}rement
%  \cite{Olympio_AAS_2010} - \cite{Venigalla_JGCD_2022}.
%}

In the field of spacecraft trajectory optimization, the Conway's
reference book~\cite{Conway_CUP_2010} presents a variety of both
analytical and numerical approaches to trajectory optimization
(see also the review papers~\cite{Shirazi_PAS_2018}
and~\cite{Chai_PAS_2019} for more recent references). This book
also contains some chapters devoted to low-thrust orbital transfer.
On this last subject, the review article~\cite{Morante_Aerospace_2021}
focuses on available numerical approaches for solving low-thrust
optimization problems, with a special emphasis of available tools
to solve such problems.

Among the literature on low-thrust trajectory design, some papers address more
specifically the issue of robustness. Indeed, guaranteeing performances such as
safety and cost under random disturbances is a major concern in a space mission
design. Such disturbances are either modeling errors (initial state, unmodeled
disturbances) or unexpected events (engine failures) or thrust realization
errors in magnitude and direction.  Different methods have been proposed to deal
with the difficulty induced by stochasticity.  A framework for primer
vector-based trajectory correction policy is proposed
in~\cite{Sidhoum_AAS_2024}.  The paper~\cite{Zavoli_JGCD_2021} investigates the
use of Reinforcement Learning for low-thrust trajectories design in presence of
noises and missed thrust events.  In~\cite{Oguri_AAS_2022}, the authors develop
a stochastic sequential convex programming approach and apply it to an
Earth-Mars transfer.  In~\cite{Greco_JGCD_2022}, the robust trajectory design
problem is formulated in the stochastic optimal control framework, more
precisely as a belief optimal control problem.  The
paper~\cite{Shimane_ASCEND_2021} deals with missed thrust events, that is, engine
failures during a planned thrusting phase, that may render the mission
unrecoverable. The authors develop a framework to assess the robustness of a
given nominal low-thrust trajectory.

Finally, three papers specifically deal with the problem of engine failure
of low-thrust propulsion spacecraft. In~\cite{Venigalla_JGCD_2022},
the authors introduce a virtual swarm method whose idea is to simultaneously
optimize the nominal trajectory together with recovery trajectories after
simulated failure. Using the swarm, they control the longest amount of time
a spacecraft may coast away from a nominal trajectory while still being
able to reach a terminal target once engine operations are resumed
(missed thrust recovery margin).
The paper~\cite{Rubinsztejn_AST_2021} investigates the use of expected
thrust fraction to design space trajectories capable of resisting missed
thrust events. The expected thrust fraction embeds the stochastic nature
of missed thrust events into a deterministic optimal control problem,
which allows for the use of computationally efficient optimal control
solvers in the design of resilient trajectories.
Finally, the title of the paper~\cite{Olympio_AAS_2010}
``Designing robust low-thrust interplanetary trajectories subject
to one temporary engine failure'' clearly indicates the interest
of the author for a problem very similar to the one we study here.
We share the same preoccupation as~\cite{Olympio_AAS_2010}:
similar modeling (stochastic optimization problem subject to probabilistic
constraints), and similar ways to tackle the problem using stochastic
approximation methods. However, we differ in the way we solve
the problem: the control structure (number of thrust and coast arcs)
is imposed in~\cite{Olympio_AAS_2010} and simplifying assumptions are
made when solving the optimal control subproblems, whereas we solve
the problem without simplification in this study. This said,
\cite{Olympio_AAS_2010} constitutes a significant contribution
to directly accounting for engine failure in an interplanetary transfer.

\subsection{Paper's content\label{ssect:intro-content}}

In~Sect.~\ref{sect:formulation} and Sect.~\ref{sect:improvement},
we present and improve the formulation
of the problem of driving the satellite in an optimal way subject
to meet the rendez-vous constraint with a given probability level.
Starting from the ``natural'' optimization problem formulated using
a probability constraint, we make use of several transformations
in order to obtain a formulation adapted to numerical treatments.
The deal is to finally obtain a stochastic optimization problem
such that both the cost function and the final target constraint
correspond to expectations.

In Sect.~\ref{sect:algorithm}, we recall the stochastic gradient method
and its application to an optimization problem subject to a constraint
in expectation. We point out the main difficulties associated with the
implementation of such an algorithm in the context of our rendez-vous
problem. Indeed a stochastic estimate of the probability constraint
gradient is not available here in a straightforward manner, and we have
to use approximation techniques in order to obtain such an estimate.
We detail the way convergence is affected by these approximations.

In Sect.~\ref{sect:realproblem}, we use the previous approximation
method to solve a real-life rendez-vous problem, and we present numerical
results illustrating how the optimal deterministic control used before any
failure is affected by the probability level associated with the target constraint.
We also illustrate the trade-off safety (probability level) versus economic
performance (fuel consumption).

Finally, we draw some conclusions in Sect.~\ref{sect:conclusion}
and we point out some obstacles, both from the theoretical
and the numerical point of view, we have to tackle for effectively
solving such problems.

\section{Problem formulation\label{sect:formulation}}

In this section, we present the satellite model and the associated
optimization problem, as well as the stochastic components that come
into play in this problem.

\subsection{Satellite model and deterministic problem\label{ssect:model}}

We consider a simplified dynamic model of a satellite, in which perturbation
forces are neglected. The equations describing the satellite evolution are
\begin{align}
  \frac{\mathrm{d}r}{\mathrm{d}t} = v
  \eqsepv \quad
  \frac{\mathrm{d}v}{\mathrm{d}t} = - \paramu \frac{r}{\norm{r}^{3}} +
  \frac{T}{m}\kappa
  \eqsepv
  \frac{\mathrm{d}m}{\mathrm{d}t} = - \frac{T}{g_{0}I_{\mathrm{sp}}} \delta
  \eqfinp
  \label{eq:dynamics}
\end{align}
The variables~$r\in\mathbb{R}^{3}$ and~$v\in\mathbb{R}^{3}$
are the inertial position and velocity of the satellite,
and~$m\in\mathbb{R}$ is its mass. The variables which allow
to control the satellite are~$\kappa\in\mathbb{R}^{3}$ (direction
cosines of the thrust) and~$\delta\in\{0,1\}$ (on-off switch of
the engine). We suppose that the control~$\delta$ is relaxed
in order to have a continuous control variable: $\delta\in[0,1]$.
These equations involve the constants~$\paramu$ (gravitational
constant), $g_{0}$ (standard acceleration), as well as the specific
impulse~$I_{\mathrm{sp}}$ and the thrust\footnote{assumed constant
rather than a function of the radius}~$T$ of the engine.

The coordinates~$(r,v)$ have several drawbacks (delicate
physical interpretation, impossible mathematical simplifications)
so that it is preferable to use the so-called equinoctial coordinates
 to model the satellite dynamics. The coordinate change is described
in Appendix~\ref{sec:satellite-details} and leads to a satellite state
variable~$x\in\mathbb{R}^{7}$
where~$(x_{1},\dots,x_{6})=(p,e_{x},e_{y},h_{x},h_{y},\ell)$ are
the equinoctial coordinates of the satellite, $x_{7}=m$ being its mass.
The satellite control variable~$u=(q,s,w)\in\mathbb{R}^{3}$ corresponds
to the radial, tangential and normal components of the control.

The deterministic control problem we aim to solve is to drive the satellite from
given initial conditions~$x_i\in \RR^7$ at a given time~$\ti$ to a known final
target, namely a position $(x_{\mathrm{f},1},\dots,x_{\mathrm{f},6})$, at a
given final time~$\tf$ while minimizing fuel consumption, that
is~$m(\ti)-m(\tf)= (x(\ti) -x(\tf))_7$. This optimal control problem is
therefore stated with a predetermined final time.  More precisely, the
deterministic optimization problem is formulated as
\begin{subequations}
  \label{pb:deterministic}
  \begin{align}
    \min_{u}
    &\quad K\bp{x(\tf)} \eqfinv \\
    &\text{s.t. } x(\ti) = x_{\mathrm{i}} \eqsepv
      \frac{\mathrm{d}x}{\mathrm{d}t}(t) = f\bp{x(t),u(t)} \eqfinv \label{eq:equinoxial}\\
    &\hphantom{\text{s.t. }} \norm{u(t)} \leq 1 \quad \forall t \in [\ti,\tf] \eqfinv \\
    &\hphantom{\text{s.t. }} C\bp{x(\tf)} = 0
      \eqfinv
  \end{align}
\end{subequations}
where mapping~$f$ is defined in Appendix~\ref{sec:satellite-details},
the mapping $K: x \in \RR^7 \mapsto (x_i -x)_7$ is the final cost
function giving the fuel consumption during the mission, and
the mapping $C: \RR^7 \to \RR^6$ describing the rendez-vous constraint
is defined by
\[
  C : x \in \RR^7 \mapsto (x_{1}-x_{\mathrm{f},1},\dots,x_{6} -x_{\mathrm{f},6}) \eqsepv
\]
where~$(x_{\mathrm{f},1},\dots,x_{\mathrm{f},6})$ are the equinoctial
coordinates corresponding to the final target~$(r_{\mathrm{f}},v_{\mathrm{f}})$.
Problem~\eqref{pb:deterministic} is called a mission, and we say that
the mission has failed if the target is not reached.

\subsection{Satellite model with engine failure during
            a time interval\label{ssect:satellite-failure}}

We want to take into account \emph{engine failures} during the satellite
mission. For that purpose we use a simplified failure model by assuming
that there is \emph{at most} one failure during a mission. Thus, a failure model
consists in an oriented pair of two times $(\tpx,\tdx)$, where $\tpx$ is
the time of the failure and $\tdx$ is the duration of the failure.

The dynamics of the satellite in the presence of failure is not different from
that given by~\eqref{eq:equinoxial}, and is characterized by the fact that the
control~$u$ is equal to 0 during the failure.  To ease the description of the
satellite dynamics, given two times $s\le s'$, we introduce a flow map
function~$\Phi_{s,s'}^{u} : \RR^{7} \to \RR^{7}$ defined by
\begin{align*}
  \Phi_{s,s'}^{u} : x \in \RR^{7}
    & \mapsto x^u(s') \\
  \text{ with }
    & x^u(s) = x \eqsepv
      \frac{\mathrm{d}x^u}{\mathrm{d}t}(t) = f\bp{x^u(t),u(t)} \eqsepv
      \forall t \in [s,s'] \eqfinp
\end{align*}

Now, assume that~$u$ is an  admissible control for Problem~\ref{pb:deterministic}.
The satellite controlled dynamics under failure has now at most three
different phases in the satellite dynamics.
\begin{enumerate}
\item During the time interval $[\ti,\tpx]$, the nominal control $u$,
  which is defined over $[\ti,\tf]$, is used. It is an open-loop control,
  that is, a control that depends only on time~$t$. The position
  of the satellite at the end of this first phase is
  \[
    x_{\tpx}= \Phi_{\ti,\tpx}^u(x_i)
    \eqfinp
  \]
\item During the time interval $[\tpx,\tpx+\tdx]$ the control is equal
  to~$0$ as the engine is not available.
  The position of the satellite at the end of the second phase is
  \[
    x_{\tpx+\tdx}= \Phi_{\tpx,\tpx+\tdx}^0(x_{\tpx})
    \eqfinp
  \]
\item
  During the time interval $[\tpx+\tdx,\tf]$, we recover engine
  control and we can try to compute a recourse control $v$
  to achieve the mission of reaching the target.
  The position of the satellite at the end of the third phase is
  \[
    x_{\tf}= \Phi_{\tpx+\tdx,\tf}^v(x_{\tpx+\tdx})
    \eqfinp
  \]
\end{enumerate}
Now, gathering the three phases of the satellite dynamics
%and noting~$(u,v,\tpx,\tdx)$ the concatenation of controls during
%these three phases,
the state of the satellite at final time
$\tf$ is given by
\begin{equation}
  \label{eq:flowuv}
  \Psi^{(u,v,\tpx,\tdx)}_{\ti,\tf}(x_i) =
  \begin{cases}
    \Phi_{\ti,\tf}^u(x_i)
    & \text{if } \tpx \ge \tf
      \eqfinv
    \\
    \Phi_{\tpx,\tf}^0 \bp{ \Phi_{\ti,\tpx}^u(x_i)}
    &\text{if } \tpx < \tf \text{ and } \tpx+\tdx \ge \tf
      \eqfinv
    \\
    \Phi_{\tpx+\tdx,\tf}^v\Bp{\Phi_{\tpx,\tpx+\tdx}^0 \bp{ \Phi_{\ti,\tpx}^u(x_i)}}
    &\text{if } \tpx < \tf \text{ and } \tpx+\tdx < \tf
      \eqfinp
  \end{cases}
\end{equation}

\subsection{Engine failures and stochastic problem\label{ssect:stochastic-problem}}

As explained in \S\ref{ssect:satellite-failure}, if an engine failure
occurs, the satellite drifts away from the deterministic optimal trajectory.
After the engine control is recovered, it is not always possible to drive
the satellite to the final target. By anticipating such possible failures
and by modifying the trajectory followed \emph{before} any such failure
occurs, one may increase the possibility of eventually reaching the target
after a failure. But such a deviation from the deterministic optimal
trajectory results in a decay of the economic performance.

We aim to formulate a new optimization problem where we need
to balance the increased probability of achieving the mission
(that is reaching the rendez-vous constraint) despite possible
failures against the expected economic performance, that is,
to quantify the price of safety one is ready to pay for.

There are different ways to accommodate these stochastic
features in an optimization problem. The first possibility
is to consider Robust Optimization (see~\cite{Bertsimas_SIAM_2011}
and the references therein for an overview of the subject).
In this approach, the uncertainty model is not stochastic,
but rather deterministic and set-based, in the sense that we seek
for a solution that is feasible for any realization of the uncertainty
in a given subset. Applied to our satellite problem, this approach aims
to consider a failure subset~$\SSoN_{\mathrm{r}}$ (for example
failures whose duration is less than or equal to a chosen value) and then
at minimizing the cost associated with the ``worst failure'' while
hitting the target for any failure in~$\SSoN_{\mathrm{r}}$:
\begin{align*}
  \min_{u} \min_{v}
  & \quad
    \Bp{\max_{(\tpx,\tdx)\in\SSoN_{\mathrm{r}}} K\Bp{  \Psi^{(u,v,\tpx,\tdx)}_{\ti,\tf}(x_i)}} \eqfinv \\
  & \text{s.t. }
    C\Bp{  \Psi^{(u,v,\tpx,\tdx)}_{\ti,\tf}(x_i)} = 0\eqsepv  \forall (\tpx,\tdx) \in\SSoN_{\mathrm{r}} \eqfinp
\end{align*}

We do not follow this direction and we rather appeal to Stochastic
Optimization (see the reference books \cite{Ruszczynski_Elsevier_2003}
and~\cite{Shapiro_SIAM_2014} for an overview). We consider a probability
space~$(\Om,\trib,\prbt)$ and we assume given two random variables $\Tp$
which is the time of the failure and $\Td$ which is the duration
of the failure. It is assumed that~$\Tp$ and~$\Td$ do not depend
on the control~$u$, which is possibly false in practice since engine
failure may depend on how one drives the engine. Note that now, the
recourse control $v$ is a random variable as it is a recourse control
which is computed knowing the two random variables~$\Tp$ and~$\Td$.
As we adopt the convention to write random variables using capital
bold letters, the recourse control $v$ is renamed as $\stratv$.
A natural criterion to be minimized in the stochastic formulation
is the expectation of the final cost, so that the optimization
problem sounds like:
\begin{equation*}
  \min_{u} \min_{\stratv}
  \Besp{K\Bp{  \Psi^{(u,\va{V},\Tp,\Td)}_{\ti,\tf}(x_i)}}
  \eqfinp
\end{equation*}

It remains to define how to take the rendez-vous constraint into account.
\begin{itemize}
\item The first possibility is to formulate the constraint
  ``almost surely'', which means that the target has to be reached
  for almost all failures according to the underlying probability
  law:
  \begin{equation*}
    C\Bp{  \Psi^{(u,\va{V},\Tp,\Td)}_{\ti,\tf}(x_i)}  =0 \eqsepv \text{\Pps[]} \eqfinp
  \end{equation*}
  This formulation is however not relevant here because we know that
  there exist realizations of the failure such that the target cannot
  be reached whatever the deterministic control~$u$ used. The admissible
  set of the optimization problem is thus empty.
\item The second possibility is to formulate the constraint
  in expectation:
  \begin{equation*}
    \Besp{C\Bp{  \Psi^{(u,\va{V},\Tp,\Td)}_{\ti,\tf}(x_i)}} =0 \eqfinp
  \end{equation*}
  Such a formulation, although mathematically attractive, is not well-suited in
  this context because meeting the rendez-vous constraint in expectation still
  leaves the possibility that the target is not reached in (almost) any failure
  scenario, and, at the very least, this does not tell how many of those
  scenarios will hit the target. This leads us to the next formulation.
\item The third possibility is to use a constraint in probability:
  in this setting, the target has not to be hit
  ``almost always'' as in the first possibility, but with a certain
  probability~$p$ whose value is given:
  \begin{equation*}
    \Bprob{C\Bp{  \Psi^{(u,\va{V},\Tp,\Td)}_{\ti,\tf}(x_i)} = 0} \geq p \eqfinp
  \end{equation*}
\end{itemize}
In the sequel, we concentrate on this last formulation in probability, which is
particularly well suited for the optimization problem under consideration in
this paper. As a matter of fact, we are here interested in a ``death or life''
process corresponding to the binary condition to reach the target or
not,\footnote{Otherwise stated, the amount the target is missed does not matter
  if the target is not reached.} and the probability constraint exactly does the
work because it counts the scenarios in which the target is reached. Note that,
while the robust approach focusses on the worst scenario in an a priori defined
subset, this formulation amounts to considering the scenarios which are on the
one hand feasible and on the other hand not too penalizing in terms of
performance as measured by the cost function.

\subsection{Initial stochastic formulation\label{ssec:form-natur}}

To summarize what was said in~\S\ref{ssect:stochastic-problem},
the first ``natural'' formulation of the optimal control problem
under consideration consists in minimizing an expectation cost
subject to a rendez-vous probability constraint:
\begin{subequations}
  \label{pb:stochnat}
  \begin{align}
    \min_{u} \min_{\stratv}
    & \quad\Besp{K\Bp{ \Psi^{(u,\va{V},\Tp,\Td)}_{\ti,\tf}(x_i)}}
      \label{eq:stochnat-crit} \\
    & \text{s.t. }
      \norm{u(t)} \leq 1 \;\;\; \forall t \in [\ti,\tf]
      \eqfinv \label{eq:stochnat-commu} \\
    & \hphantom{\text{s.t. }} \|\va{V}(t)\| \leq 1 \;\;\; \forall t \in [\Tp+\Td,\tf]
      \eqfinv \label{eq:stochnat-commv}\\
    &\hphantom{\text{s.t. }} \Bprob{C\Bp{ \Psi^{(u,\va{V},\Tp,\Td)}_{\ti,\tf}(x_i)} = 0} \geq p
      \eqfinp \label{eq:stochnat-cible}
  \end{align}
\end{subequations}

Our main goal is to numerically solve this optimization problem.
But Formulation~\eqref{pb:stochnat} is not adapted to that purpose.
First, the cost function takes all possible failures into account
(by mean of the expectation operator), whereas we  are in fact not
interested in failures such that the satellite misses the target.
Second, the probability constraint cannot be treated as is, and
we need to find a more practical expression in order to numerically
deal with it. Indeed, probability constraints raise important
mathematical difficulties, such as the lack of convexity or connectedness
of the feasible subset they induce. Those convexity or connectedness (or emptiness) properties depend of course on the properties of the function
entering the constraint, but also on the probability distribution
of the random variable and on the level~$p$ of probability required.
One may refer to~\cite{Prekopa_Kluwer_1995} and~\cite{Prekopa_Elsevier_2003}
for a general presentation of probability constraints,
to~\cite{Henrion_JOTA_2002} for connectedness properties,
to~\cite{Henrion_COA_2008} for convexity properties and
to~\cite{Henrion_MP_2004} for the question of stability
(see also~\cite[Chap.4]{Shapiro_SIAM_2014} for a review
on optimization under probabilistic constraints).
In this paper, we are mainly interested in the numerical difficulties
arising when dealing with probability constraints. Thus we do not focus
on theoretical questions and we assume that all the problems under
consideration have the required properties when using numerical algorithms:
connectedness and convexity of the admissible set, continuity and differentiability of the constraints and existence of saddle points
for the Lagrangian associated to the problem.

In the next section, we modify and improve the formulation
in order to obtain a numerically tractable optimization problem.

\section{Formulation improvement\label{sect:improvement}}

We now describe the improvements to be made to Problem~\eqref{pb:stochnat}
to obtain an operational formulation.

\subsection{Expectation formulation\label{ssec:improv-expectation}}

\subsubsection{Missing the target should not contribute to the cost function}

When a failure occurs, it may be impossible for some cases to reach the
target. For the other cases, depending on the required level~$p$ in the
probability constraint~\eqref{eq:stochnat-cible}, one may decide not to reach
the target even if possible. For all the failure scenarios such that the target
is missed, the optimal control consists in not restarting the engine at the end
of the failure because, as long as these scenarios contribute to the expected
cost~\eqref{eq:stochnat-crit}, this is the way to reduce the fuel consumption.
But we are in fact not interested in what happens whenever the mission fails, so
that these scenarios yield an artificially good cost. Otherwise stated, a
failure scenario, namely when the satellite misses the target and thus does not
contribute to the probability constraint, should not be taken into account in
the cost function. Thus we have to replace the original expected cost
in~\eqref{eq:stochnat-crit} by the expected cost \emph{conditionally} to
reaching the target, so that the problem becomes
\begin{subequations}
  \label{pb:stochnat-cond}
  \begin{align}
    \min_{u} \min_{\stratv}
    & \quad \Bespc{K\Bp{\Xtf }}{C\bp{\Xtf}=0}
      \label{eq:stochnat-crit-cond} \\
    & \text{s.t. } \eqref{eq:stochnat-commu} \eqsepv
                   \eqref{eq:stochnat-commv}\eqsepv \\
    &\hphantom{\text{s.t. }}
     \Bprob{C\Bp{ \Psi^{(u,\va{V},\Tp,\Td)}_{\ti,\tf}(x_i)} = 0} \geq p
     \eqfinp \label{eq:stochnat-cible-cond}
  \end{align}
\end{subequations}

\subsubsection{Probability as an expectation}

Using standard properties of conditional expectation we
rewrite Problem~\ref{pb:stochnat-cond} as follows.

\begin{subequations}
  \label{pb:stochcond}
  \begin{align}
    \min_{u} \min_{\stratv}
    &\quad
      \frac{\Besp{K\bp{\Xtf} \times
      \fctI\bp{\big\|C\bp{\Xtf}\big\|}}}
      {\Besp{\fctI\bp{\big\|C\bp{\Xtf}\big\|}}}
      \eqfinv \label{eq:stochcond-crit} \\
    & \text{s.t. }
      \eqref{eq:stochnat-commu}\eqsepv \eqref{eq:stochnat-commv} \eqsepv \\
    & \hphantom{\text{s.t. }}
      \Besp{\fctI\bp{\big\|C\bp{\Xtf}\big\|}} \geq p
      \eqfinv\label{eq:stochcond-cible}
  \end{align}
\end{subequations}
where the real-valued indicator function~$\fctI$ is defined,
for non negative values of~$y$, by
\begin{equation}
  \label{eq:indicator-function}
  \fctI(y)=
  \begin{cases}
    1 & \text{if\ } y=0, \\
    0 & \text{otherwise}
        \eqfinp
  \end{cases}
\end{equation}

\begin{remark}
  The indicator function we introduced both in the cost function
  and in the probability constraint is highly discontinuous.
  Note however that the function~$\fctI$ is up to now inside
  an expectation so that we may recover smooth expressions once
  expectation is achieved.
\end{remark}

\subsubsection{Expectation based reformulation}

In Problem~\eqref{pb:stochcond}, the denominator
in the ratio defining the cost function is identical to the left-hand side
of the constraint. This specificity allows us to use Lemma~\ref{lem:useful}
in Appendix~\ref{sec:form-lemma}, so that it is possible to replace
the ratio defining the criterion~\eqref{eq:stochcond-crit} by its numerator,
provided that we a posteriori check a condition on the optimal multiplier
associated with Constraint~\eqref{eq:stochcond-cible}.
Instead of~\eqref{pb:stochcond}, we aim at solving the following problem
for which both cost and constraint functions are expressed as expectations:
\begin{subequations}
  \label{pb:stochexpect}
  \begin{align}
    \min_{u} \min_{\stratv}
    &\quad
      \Besp{K\bp{\Xtf} \times \fctI\bp{\big\|C\bp{\Xtf}\big\|}}
      \label{eq:stochexpect-crit} \\
    & \text{s.t. }
      \eqref{eq:stochnat-commu}\eqsepv \eqref{eq:stochnat-commv}\eqsepv\\
    & \hphantom{\text{s.t. }}
      \Besp{\fctI\bp{\big\|C\bp{\Xtf}\big\|}} \geq p \eqfinp
      \label{eq:stochexpect-cible}
  \end{align}
\end{subequations}
With that reformulation, we have obtained a stochastic optimization
problem formulation better suited to numerical resolution than
the previous formulation~\eqref{pb:stochcond}, in the sense where
both the criterion and the constraint correspond to expectations.

\subsubsection{Introducing the no-failure scenario\label{ssec:form-nofailure}}

The event $\ba{ \Tp \geq\tf} \subset \Om$
is called the ``no-failure scenario'', because it corresponds
to the set of failure scenarios that occur after the end of the mission.
We denote by~$\pif$ the probability of that event, that is:
\begin{equation*}
  \pif = \bprob{\ba{ \Tp \geq\tf}} \eqfinp
\end{equation*}
For every~$\omega \in\ba{ \Tp \geq\tf}$, there is no need for a closed-loop
control~$\va{V}$. Indeed, on the event $\ba{ \Tp \geq\tf}$, the state
of the satellite at final time $\tf$ only depends on the nominal control~$u$.

For classical missions, we do hope that the probability~$\pif$
of the no-failure scenario is rather large (close to~$1$). Moreover,
we expect the stochastic Problem~\eqref{pb:stochexpect} to have
admissible solutions for some probability levels~$p$ higher than
the probability $\pif$ of the ``no-failure scenario''.
Therefore, we make the following assumption on $p$ and $\pif$.

\begin{assumption}
  \label{as:probalevel}
  The final time~$\tf$ and the probability level~$p$
  of Problem~\eqref{pb:stochexpect} are such that
  \begin{equation*}
    p  \geq \pif > \frac{1}{2} \eqfinp
    \label{eq:assumption-pi}
  \end{equation*}
\end{assumption}

Under Assumption~\ref{as:probalevel}, we give in Lemma~\ref{lem:u-target}
an equivalent formulation of Problem~\eqref{pb:stochexpect}.
Moreover, one deduces from Lemma~\ref{lem:u-target}
that the optimal open-loop control used
prior to any failure must be such that the associated state
trajectory meets the rendez-vous constraint in the scenario
without failure.
\begin{lemma}
  \label{lem:u-target}
  Under Assumption~\ref{as:probalevel}, Problem~\eqref{pb:stochexpect}
  is equivalent to the following Problem~\eqref{pb:stochexpect-nofail}
  \begin{subequations}
    \label{pb:stochexpect-nofail}
    \begin{align}
      \min_{u} \min_{\stratv}
      &\quad
        \pif \, K\bp{\Phi^u_{\ti,\tf}(x_i)} +
        \bp{1-\pif}
        \Bespc{K_I\bp{\Xtf}}{\Tp <\tf}\eqsepv
        \label{eq:stochexpect-crit-nofail}
      \\
      & \text{s.t. }
        \eqref{eq:stochnat-commu}\eqsepv \eqref{eq:stochnat-commv}\eqsepv \\
      & \hphantom{\text{s.t. }}\pif \, C\bp{\Phi^u_{\ti,\tf}(x_i)} = 0 \eqsepv
        \label{eq:stochexpect-nofail-u}
      \\
      & \hphantom{\text{s.t. }}
        \bp{1-\pif} \Bespc{\fctI\bp{\big\|C\bp{\Xtf}\big\|}}{\Tp <\tf} \geq p-\pif
        \label{eq:stochexpect-nofail-V}
        \eqfinv
    \end{align}
    where the function $K_I$ is defined by $K_I(\cdot):= K(\cdot) \times \fctI\bp{\big\|C\bp{\cdot}\big\|}$.
  \end{subequations}
\end{lemma}

\begin{proof}
  As a preliminary result, for any measurable mapping $\varphi$,
  we have that\footnote{with, for a set~$A$, the notation
  $\mathds{1}_{A}(a) = 1$ if~$a\in A$ and~$\mathds{1}_{A}(a) = 0$
  otherwise}
  \begin{align}
    &\Besp{\varphi\Bp{\Xtf}} \nonumber \\
    &\hspace{1cm} = \Besp{\varphi\Bp{\Xtf}\mathds{1}_{\na{\Tp \geq\tf}} +\varphi\Bp{\Xtf}\mathds{1}_{\na{\Tp < \tf}}}
      \nonumber \\
    &\hspace{1cm} = \Besp{\varphi\Bp{\Phi^u_{\ti,\tf}(x_i)}\mathds{1}_{\na{\Tp \geq\tf}} +\varphi\Bp{\Xtf}\mathds{1}_{\na{\Tp < \tf}}}
      \nonumber \\
    &\hspace{1cm} = \pif \varphi\Bp{\Phi^u_{\ti,\tf}(x_i)} + (1-\pif)\Bespc{\varphi\Bp{\Xtf}}{{\na{\Tp < \tf}}}
      \label{eq:esp-Tp}
      \eqfinp
  \end{align}
  Now, applying Equality~\eqref{eq:esp-Tp} with $\varphi(\cdot)= \fctI\bp{\big\|C\bp{\cdot}\big\|}$, we obtain that
  Constraint~\eqref{eq:stochexpect-cible} can be equivalently rewritten as
  \begin{equation}
    \label{eq:stochexpect-cibis}
    \pif \, \fctI\bp{\big\|C\bp{\Phi^u_{\ti,\tf}(x_i)}\big\|} + \bp{1-\pif}
    \Bespc{\fctI\bp{\big\|C\bp{\Xtf}\big\|}}{\Tp <\tf} \geq p \eqfinp
  \end{equation}
  Then, we prove that~\eqref{eq:stochexpect-cibis} is equivalent
  to~\eqref{eq:stochexpect-nofail-u} and~\eqref{eq:stochexpect-nofail-V}.
  \begin{itemize}
  \item First, it is immediate that the implication
  $\eqref{eq:stochexpect-nofail-u} \wedge \eqref{eq:stochexpect-nofail-V}
   \implies \eqref{eq:stochexpect-cibis}$ holds true. Indeed,
  assuming~\eqref{eq:stochexpect-nofail-u}, we
  have that $\fctI\bp{\big\|C\bp{\Phi^u_{\ti,\tf}(x_i)}\big\|}=1$ and then
  Inequality~\eqref{eq:stochexpect-nofail-V} gives~\eqref{eq:stochexpect-cibis}.
  \item Second, we prove the reverse implication
  $\eqref{eq:stochexpect-cibis} \implies \eqref{eq:stochexpect-nofail-u}
   \wedge \eqref{eq:stochexpect-nofail-V}$ by contradiction. Assume
  that~\eqref{eq:stochexpect-cibis} is satisfied and
  that~$C\bp{\Phi^u_{\ti,\tf}(x_i)} \neq 0$, that is,
  $\fctI\bp{\|C(\Phi^u_{\ti,\tf}(x_i))\|}=0$. Then we obtain
  \begin{equation}
    \bp{1-\pif} \Bespc{\fctI\bp{\big\|C\bp{\Xtf}\big\|}}{\Tp<\tf} \geq p \eqfinp
    \label{eq:ineq-p}
  \end{equation}
  As the left-hand side of Equation~\eqref{eq:ineq-p} is smaller
  that~$1-\pif$ and the right-hand side is greater that~$\pif$
  using Assumption~\ref{eq:assumption-pi}, we obtain that
  $1-\pif \ge \pif$ which contradicts Assumption~\ref{eq:assumption-pi}.
  We have obtained that~\eqref{eq:stochexpect-cibis}
  implies~\eqref{eq:stochexpect-nofail-u} and thus immediately~\eqref{eq:stochexpect-nofail-V}.
  \end{itemize}
  Finally, we consider the cost in~\eqref{eq:stochexpect-crit}.
  Applying Equality~\eqref{eq:esp-Tp} with $\varphi  = K_I$,
  the cost function in~\eqref{eq:stochexpect-crit} can be rewritten as
  \begin{subequations}
    \begin{align}
      &\Besp{K\bp{\Xtf} \times \fctI\bp{\big\|C\bp{\Xtf}\big\|}} \nonumber
      \\
      &\hspace{1cm} =
        \pif \, K_I\bp{\Phi^u_{\ti,\tf}(x_i)} +
        \bp{1-\pif}
        \Bespc{K_I\bp{\Xtf}}{\Tp <\tf}\eqsepv
        \tag{by~\eqref{eq:esp-Tp}}
      \\
      &\hspace{1cm} =
        \pif \, K\bp{\Phi^u_{\ti,\tf}(x_i)} +
        \bp{1-\pif}
        \Bespc{K_I\bp{\Xtf}}{\Tp <\tf}
        \tag{using~\eqref{eq:stochexpect-nofail-u}}
        \eqsepv
    \end{align}
  \end{subequations}
  which gives~\eqref{eq:stochexpect-crit-nofail}. The proof is complete.
\end{proof}

Introducing the no-failure scenario in the formulation of the stochastic
optimization problem allows us to use the conditional probability
law~$\bprobc{\cdot}{\Tp <\tf}$ instead of the initial probability
law~$\bprob{\cdot}$.
As we aim at implementing a stochastic gradient algorithm for solving
the problem (see forthcoming Sect.~\ref{sect:algorithm}), it would not
be wise to draw failure scenarios according to the initial law~$\prbt$,
since the single no-failure scenario alone carries a large probability
mass~$\pif$. By using the conditional law~$\bprobc{\cdot}{\Tp<\tf}$,
a failure occurring during the mission is drawn at each iteration
of the stochastic gradient algorithm.

\subsection{Final formulation obtained by smoothing indicator function~$\fctI$}

A major challenge in the formulation~\eqref{pb:stochexpect-nofail} is how to
deal with the indicator function~$\fctI$ that isn't even continuous.  Using the
notion of mollifier introduced in~\cite{Ermoliev_SIAM_1995}, we replace in
Problem~\eqref{pb:stochexpect-nofail} the function~$\fctI$ by a smoother (in
fact continuous) function $\fctI_{r}$ defined for non negative values of~$y$
and depending on a parameter~$r>0$, namely.
\begin{equation}
  \label{eq:smooth-indicator-function}
  \fctI_{r}(y)= \max\Ba{0,1-\frac{y}{r}} \eqfinv
\end{equation}
where the approximation~$\fctI_{r}$ of~$\fctI$ improves as~$r$
decreases. Finally, we consider the following smoothed
Problem~\eqref{pb:finalsmooth} derived from
Problem~\eqref{pb:stochexpect-nofail}:
\begin{subequations}
  \label{pb:finalsmooth}
  \begin{align}
    \min_{u \in \mathcal{U}} \min_{\stratv \in \mathcal{V}}
    &\quad
      \pif \, K\bp{\Phi^u_{\ti,\tf}(x_i)} +
      \bp{1-\pif}
      \Bespc{K_{\fctI_r}\bp{\Xtf}}{\Tp <\tf}\eqsepv
      \label{eq:finalsmooth-crit}
    \\
    & \text{s.t. }\pif \, C\bp{\Phi^u_{\ti,\tf}(x_i)} = 0 \eqsepv
      \label{eq:stochexpect-nofail-u-r}
    \\
    & \hphantom{\text{s.t. }} p - \pif \; - \; \bp{1-\pif} \:
      \Bespc{\fctI_{r}\bp{\big\|C\bp{\Xtf}\big\|}}{\Tp<\tf}
      \leq 0 \eqfinp
      \label{eq:finalsmooth-cibl2}
    \\[0.5cm]
    &\text{ with  }{\mathcal{U}} =\nset{u}{ \forall t \in [\ti,\tf]\eqsepv \norm{u(t)} \leq 1}
      \eqfinv
    \nonumber \\
    &\hphantom{\text{ with  }}{\mathcal{V}}
      =\nset{\va{V}}{\forall t \in [\Tp+\Td,\tf]\eqsepv \|\va{V}(t)\| \leq 1}
      \nonumber
  \eqfinp
  \end{align}
\end{subequations}

\section{Stochastic gradient algorithm\label{sect:algorithm}}

Formulation~\eqref{pb:finalsmooth} consists in minimizing the expected
cost~\eqref{eq:finalsmooth-crit} subject to the constraint in
expectation~\eqref{eq:finalsmooth-cibl2}.  It is thus well suited for applying a
stochastic gradient algorithm, namely the Arrow-Hurwicz algorithm, for finding a
saddle point of the Lagrangian of Problem~\eqref{pb:finalsmooth}:
see~\cite{Culioli_SIAM_1990} for stochastic gradient methods with deterministic
constraints, \cite{Culioli_CRAS_1995} and~\cite{Culioli_RI_1994} when dealing
with constraints in expectation in the sub-differentiable case (see
also~\cite[Chap~10]{Carpentier_Springer_2017} for a more recent reference, in
French). In this section, we develop the practical implementation on the
Arrow-Hurwicz algorithm. As already mentioned, we do not focus on theoretical
questions and we assume that using such an algorithm can be justified in the
present setting. The main assumption is in fact the existence of a saddle point
for Problem~\eqref{pb:finalsmooth}, a difficult issue in the framework of
probabilistic constraints.

\subsection{Lagrangian formulation\label{ssect:algo-lagrang}}

Denoting by~$\mu \in \RR^{+}$ the multiplier associated with
Constraint~\eqref{eq:finalsmooth-cibl2}, the dual optimization
problem we have to solve is
\begin{subequations}
  \label{pb:finalsmooth-dual-orig}
  \begin{align}
    \max_{\mu \geq 0} \min_{u \in \mathcal{U}\,, \stratv \in \mathcal{V}}
    &\quad
      \pif \, K\bp{\Phi^u_{\ti,\tf}(x_i)} +
      \bp{1-\pif}
      \Bespc{{K_{\fctI_r}^{\mu}\bp{\Xtf}}}{\Tp <\tf}
      \eqsepv
      \label{eq:finalsmooth-dual-orig-crit}
    \\
    & \text{s.t. }\pif \, C\bp{\Phi^u_{\ti,\tf}(x_i)} = 0
      \label{eq:stochexpect-dual-orig-nofail-u-r}
      \eqfinv
  \end{align}
\end{subequations}
with the notation $K_{\fctI_r}^{\mu}(\cdot)=(K(\cdot)-\mu)\times\fctI_r(\cdot)$.
The minimization with respect to the closed-loop strategy $\va{V}$
and the conditional expectation can be interchanged
(see \cite[Theorem 14.60]{Rockafellar_Springer_1998} for mathematical
details): loosely speaking, there are as many closed-loop controls
as scenarios, so that the minimization with respect to the closed-loop
strategy is in fact achieved scenario by scenario. Therefore Problem~\eqref{pb:finalsmooth-dual-orig} is equivalent to the
following dual problem:
\begin{subequations}
  \label{pb:finalsmooth-dual}
  \begin{align}
    \max_{\mu \geq 0} \min_{u \in \mathcal{U}}
    &\quad
      \pif \, K\bp{\Phi^u_{\ti,\tf}(x_i)} +
      \bp{1-\pif}
      \Bespc{ \min_{\stratv \in \mathcal{V}} {K_{\fctI_r}^{\mu}\bp{\Xtf}}}{\Tp <\tf}
      \eqsepv
      \label{eq:finalsmooth-dual-crit-ter}
    \\
    & \text{s.t. }\pif \, C\bp{\Phi^u_{\ti,\tf}(x_i)} = 0
      \label{eq:stochexpect-dual-nofail-u-r-ter}
      \eqfinp
  \end{align}
\end{subequations}
This dual problem incorporates an optimization problem in~$\va{V}$, whose
optimal value depends on the no-failure scenario only by means of the state
value ~$x^u_{\Tp}= \Phi^u_{\ti,\Tp}(x_i)$.  Indeed, using the fact that we only
consider events in $\na{\Tp < \tf}$, the value of function~$\Psi$
in~Equation~\eqref{eq:flowuv} is rewritten as
\begin{subequations}
\begin{align}
  \Psi^{(u,\va{V},\Tp,\Td)}_{\ti,\tf}
  & (x_i)
    = \Gamma_{\Tp,\tf}^{\va{V},\Td}\bp{ \Phi_{\ti,\Tp}^u(x_i)}
  \\
  & \text{with }
    \Gamma_{\Tp,\tf}^{\va{V},\Td}: x \mapsto
    \begin{cases}
    \Phi_{\Tp,\tf}^0 \np{ x}
    &\text{if } \Tp+\Td \ge \tf
      \eqfinv
    \\
    \Phi_{\Tp+\Td,\tf}^{\va{V}} \bp{\Phi_{\Tp,\Tp+\Td}^0 \np{x}}
    &\text{if } \Tp+\Td < \tf
      \eqfinv
  \end{cases}
\end{align}
\end{subequations}
and we rewrite Problem~\eqref{pb:finalsmooth-dual} as
\begin{subequations}
  \label{pb:finalsmooth-dual1}
  \begin{align}
    \max_{\mu \geq 0} \min_{u \in \mathcal{U}}
    &\quad
      \pif \, K\bp{\Phi^u_{\ti,\tf}(x_i)} +
      \bp{1-\pif}
      \Bespc{ W_{r}\bp{ \Phi_{\ti,\Tp}^u(x_i) , \Tp,\Td,\mu}}{\Tp <\tf}
      \eqsepv
      \label{eq:finalsmooth-crit-1}
    \\
    & \text{s.t. }\pif \, C\bp{\Phi^u_{\ti,\tf}(x_i)} = 0
      \label{eq:finalsmooth-nofail-u-r-1}
      \eqfinv
  \end{align}
\end{subequations}
where the mapping $W_{r}\bp{x,\tpx,\tdx,\mu}$ is defined by
\begin{equation}
  \label{eq:inner}
  W_{r}\bp{x,\tpx,\tdx,\mu} =
  \min_{v}
  \bp{K\bp{ \Gamma_{\tpx,\tf}^{v,\tdx} \np{x}}-\mu} \fctI_{r}\bp{\big\|C\bp{ \Gamma_{\tpx,\tf}^{v,\tdx} \np{x}}\big\|}
  \eqfinp
\end{equation}
Problem~\eqref{eq:inner} is hereafter called the \emph{internal problem}.
Assuming that the function~$W_{r}$ can be computed for any given
sample of the random variables~$\Tp$ and~$\Td$, we observe that the
\emph{external} max-min problem~\eqref{pb:finalsmooth-dual1} now
involves only deterministic variables~$\mu$ and~$u$.

\subsection{Stochastic Arrow-Hurwicz algorithm\label{ssect:stochastic-arrow}}

The Arrow-Hurwicz algorithm consists in alternating gradient
steps performed on the primal decision variable and on the
dual variable associated with a dualized constraint.
The algorithm, first introduced in~\cite{Arrow_SUP_1958} in the
deterministic framework,  has been extended to the stochastic case
with deterministic constraint (see~\cite{Culioli_SIAM_1990}) and
with constraint in expectation (see~\cite{Culioli_CRAS_1995},
proofs in~\cite{Culioli_RI_1994}).
In the latter case, one considers the problem
\begin{equation*}
\min_{u \in U\ad} \besp{j(u,\boldsymbol{\xi})} \quad \text{s.t. }
\besp{\theta(u,\boldsymbol{\xi})} \leq 0 \eqfinv
\end{equation*}
where~$u$ is a deterministic control (open-loop)
and~$\boldsymbol{\xi}$ is a random variable with given probability law.
Then the~$k$-th iteration of the stochastic Arrow-Hurwicz algorithm
consists in drawing a sample~$\xi^{k}$ of~$\boldsymbol{\xi}$ and
updating the primal variable~$u$ and the dual variable~$\mu$
by (sub-)gradients steps, namely,
\begin{align*}
 & u^{k+1} = \proj{U\ad}{u^{k}-\epsilon_{u}^{k} \nabla_{u}
   \bp{j(u^{k},\xi^{k})+{\mu^{k}}\transp \theta(u^{k},\xi^{k})}}
   \eqfinv \\
 & \mu^{k+1} = \max \ba{0,\mu^{k}+\epsilon_{\mu}^{k}\theta(u^{k},\xi^{k})}
   \eqfinv
\end{align*}
where~$\epsilon_{u}^{k}$ and~$\epsilon_{\mu}^{k}$ are infinite sequences
of numbers such that
\begin{equation*}
\epsilon^{k} > 0 \eqsepv
\sum_{k=0}^{+\infty} \epsilon^{k} = +\infty \eqsepv
\sum_{k=0}^{+\infty} \bp{\epsilon^{k}}^{2} < +\infty \eqfinp
\end{equation*}

Here we consider Problem~\eqref{pb:finalsmooth} and its dual reformulation~\eqref{pb:finalsmooth-dual1}.
At iteration~$k$ of the stochastic Arrow-Hurwicz algorithm,
we have at our disposal current values~$u^{k}$ for
the control, $x^{k}$ for the state of the no-failure
scenario, $\mu^{k}$ for the multiplier, $r^{k}$ for the smoothing
coefficient, $\epsilon_{u}^{k}$ and~$\epsilon_{\mu}^{k}$ for
the gradient step lengths. The iteration consists
of the following steps.
\begin{enumerate}
\item[(1)]
Draw a failure scenario~$\SoN^{k}=(\tpxk,\tdxk)$ according
to the conditional probability law
$\nprobc{\cdot}{\tpxk\leq\tf}$.
\item[(2)]
Compute the partial (sub-)gradient
$\nabla_{x}W_{r^{k}}\bp{x^{k}(\tpxk),\tpxk,\tdxk,\mu^{k}}$
with respect to~$x$ and the partial (sub-)gradient
$\nabla_{\mu}W_{r^{k}}\bp{x^{k}(\tpxk),\tpxk,\tdxk,\mu^{k}}$
with respect to~$\mu$ of the function~$W_{r^{k}}$ (defined in Equation~\eqref{eq:inner}), optimal
cost of the inner optimization problem.
\item[(3)]
In backward time, compute the adjoint state~$\lambda^{k}(\cdot)$
on~$[\tpxk,\tf]$:
\begin{equation*}
\lambda^{k}(\tf) = \pif \nabla K\bp{x^{k}\np{\tf}} \eqsepv
\frac{\mathrm{d}\lambda^{k}}{\mathrm{d}t}(t) =
\nabla_{x}f(x^{k}(t),u^{k}(t))\transp \lambda^{k}(t) \eqfinp
\end{equation*}
\item[(4)]
At time~$\tpxk$, incorporate the jump induced by the cost
term~$(1-\pif)W_{r^{k}}$:
\begin{equation*}
\lambda_{-}^{k}(\tpxk) = \lambda^{k}(\tpxk) + (1-\pif)
\nabla_{x} W_{r^{k}}\bp{x^{k}(\tpxk),\tpxk,\tdxk,\mu^{k}} \eqfinp
\end{equation*}
\item[(5)]
In backward time, compute the adjoint state~$\lambda^{k}(\cdot)$
on~$[\ti,\tpxk]$ starting from the initial
value~$\lambda_{-}^{(k)}(\tpxk)$.
\item[(6)]
Compute the partial (sub-)gradient
$\nabla_{u}H\bp{x^{k}(\cdot),u^{k}(\cdot),\lambda^{k}(\cdot)}$
with respect to~$u$ of the Hamiltonian~$H(x,u,\lambda) = \lambda\transp f(x,u)$,
and update the control~$u$ by performing a descent
gradient step projected on the intersection of the unit
ball~$S_{1} = \ba{u \mid \|u(t)\| \leq 1 , \forall t}$
and the subset~$C_{\ad} = \ba{u \mid C(x(\tf))=0}$ of controls
to reach the final target:
\begin{equation*}
u^{k+1}(\cdot) = \proj{S_{1} \cap C_{\ad}}{u^{k}(\cdot) - \epsilon_{u}^{k}
\nabla_{u}H\np{x^{k}(\cdot),u^{k}(\cdot),\lambda^{k}(\cdot)}} \eqfinp
\end{equation*}
\item[(7)]
Update the dual variable~$\mu$ by performing a gradient ascent step:
\begin{equation*}
\mu^{k+1} = \proj{\mathbb{R}^{+}}
{\mu^{(k)} + \epsilon_{\mu}^{k} \bp{p-\pif +
(1-\pif)\nabla_{\mu}W_{r^{k}}\bp{x^{k}(\tpxk),\tpxk,\tdxk,\mu^{k}}}}
\eqfinp
\end{equation*}
\item[(8)]
Compute in forward time the state~$x^{k+1}(\cdot)$ on~$[\ti,\tf]$:
\begin{equation*}
x^{k+1}(\ti) = x_{\mathrm{i}} \eqsepv
\frac{\mathrm{d}x^{k+1}}{\mathrm{d}t}(t) = f(x^{k+1}(t),u^{k+1}(t)) \eqfinp
\end{equation*}
\end{enumerate}
The main outputs of the algorithm are the optimal state
trajectory~$x\opt$ and control trajectory~$u\opt$
for the no-failure scenario.

\begin{remark}
\label{rm:projection}
As will be explained in~\S\ref{ssec:inner-problem},
the target constraint~$C(x(\tf))=0$ is treated by duality
when solving the inner problem in order to compute~$W_{r^{k}}\bp{x^{k}(\tpxk),\tpxk,\tdxk,\mu^{k}}$
and its partial (sub-)gradients. By contrast, in the stochastic
Arrow-Hurwicz algorithm described above, the target
constraint~$C(x(\tf))=0$ associated with the no-failure scenario
is treated by projection. The reason for that is that dualizing
this last target constraint has led in our numerical experiments
to an instable behavior of the stochastic Arrow-Hurwicz algorithm
for high probability level~$p$. Note that the projection
on the intersection~$S_{1} \cap C_{\ad}$ involves the resolution
of an optimal control problem. Indeed, using the notation
$u^{k+\frac{1}{2}}(\cdot)=u^{k}(\cdot)-\epsilon_{u}^{k}\nabla_{u}
 H\np{x^{k}(\cdot),u^{k}(\cdot),\lambda^{k}(\cdot)}$,
this projection amounts to solving:
\begin{subequations}
\label{pb:projection}
\begin{align}
  \min_{u} \frac{1}{2}
  & \int_{\ti}^{\tf} \bnorm{u(t)-u^{k+\frac{1}{2}}(t)}^{2} \mathrm{d}t \\
  \text{s.t. }
  & x(\ti) = x_{\mathrm{i}} \eqsepv
    \frac{\mathrm{d}x}{\mathrm{d}t}(t) = f\bp{x(t),u(t)} \eqfinv  \\
  & \norm{u(t)} \leq 1 \;\;\; \forall t \in [\ti,\tf] \eqfinv \\
  & C\bp{x(\tf)}=0 \eqfinp
\end{align}
\end{subequations}
This problem is well conditioned and can be easily solved using
an adapted optimal control algorithm such as the one used to solve
the inner problem (see~\S\ref{ssec:inner-problem}).
\end{remark}

In summary, one iteration of the stochastic Arrow-Hurwicz algorithm
aims at solving two deterministic optimal control problems:
\begin{itemize}
\item Problem~\eqref{eq:inner} for computing the optimal value
$W_{r^{k}}\bp{x^{k}(\tpxk),\tpxk,\tdxk,\mu^{k}}$ and the associated
derivatives,
\item Problem~\eqref{pb:projection} for computing the projection
on~$S_{1} \cap C_{\ad}$.
\end{itemize}
In the next section, we detail the way Problem~\eqref{eq:inner}
is treated.

\subsection{A heuristic for solving the inner Problem~\eqref{eq:inner}\label{ssec:inner-problem}}

As just explained in \S\ref{ssect:stochastic-arrow}, at each iteration
of the stochastic Arrow-Hurwicz algorithm, we have to compute
$W_{r^{k}}\bp{x^{k}(\tpxk),\tpxk,\tdxk,\mu^{k}}$ which is the optimal
value of the deterministic inner optimal control problem~\eqref{eq:inner},
as well as its partial (sub-)gradient. To perform that task,
given $(x,\tpx,\tdx,\mu)= (x^{k}(\tpxk),\tpxk,\tdxk,\mu^{k})$,
we have to solve the optimal control problem:
\begin{subequations}
  \label{pb:inner-problem}
  \begin{align}
    & \min_{v} \bp{K\bp{x(\tf)}-\mu} \fctI_{r}\bp{\big\|C\bp{x(\tf)}\big\|}
      \label{eq:inner-crit} \eqfinv \\
    &\text{s.t. }
      x(\tpx) = x \eqsepv
      \frac{\mathrm{d}x}{\mathrm{d}t}(t) = f\bp{x(t),v(t)}
      \label{eq:inner-dyn} \eqfinv \\
    &\hphantom{\text{s.t. }} v(t) = 0 \quad \forall t \in [\tpx,\tpx+\tdx]
      \notag \eqfinv \\
    &\hphantom{\text{s.t. }} \|v(t)\| \leq 1 \quad \forall t \in [\tpx+\tdx,\tf]
      \label{eq:inner-bound} \eqfinp
  \end{align}
\end{subequations}
Problem~\eqref{pb:inner-problem} is difficult to tackle numerically
because, for a control~$v$ that does not bring the final state constraint
into a ball of radius~$r$ (assumed to be or to become small), the associated
cost is zero, as is its sensitivity to the control.
% Nevertheless,
% this problem can be viewed as an optimal control problem in which
% fuel consumption is minimized under the constraint that the final
% state stays within the ball of radius~$r$, namely,
% \begin{align*}
%   \min_{v} \; & K\bp{x(\tf)}-\mu \eqfinv \\
%               &\text{s.t. \eqref{eq:inner-dyn}, \eqref{eq:inner-bound}, and} \\
%               &\hphantom{\text{s.t. }} \bnorm{C\bp{x(\tf)}} \leq r \eqfinp
% \end{align*}
For that reason, in order to obtain a numerical solution of Problem~\eqref{pb:inner-problem}, we solve the optimal control
problem of minimizing the final cost under the constraint
of hitting the target exactly, that is,
\begin{subequations}
  \label{pb:innerequiv-problem}
  \begin{align}
    & \min_{v} K\bp{x(\tf)}-\mu
      \label{eq:innerequiv-crit} \eqfinv \\
    &\text{s.t. \eqref{eq:inner-dyn}, \eqref{eq:inner-bound}, and}\\
    &\hphantom{\text{s.t. }} C\bp{x(\tf)} = 0
      \label{eq:innerequiv-cible}
      \eqfinv
  \end{align}
\end{subequations}
using the deterministic Arrow-Hurwicz algorithm~\cite{Arrow_SUP_1958},
where the final constraint~\eqref{eq:innerequiv-cible} is dualized using
an \emph{augmented Lagrangian}\footnote{that is, additional duality
terms $\nu\transp C\bp{x(\tf)} + \frac{c}{2}\bnorm{C\bp{x(\tf)}}^{2}$}
in the algorithm. The deterministic Arrow-Hurwicz algorithm produces
a sequence~$\ba{\delta^{\kappa}}_{\kappa\in\mathbb{N}}$ where $\delta^{\kappa}$
is the deviation~$C\bp{x(\tf)}$ from target at iteration~$\kappa$, and
a sequence~$\ba{\upsilon^{\kappa}}_{\kappa\in\mathbb{N}}$
where $\upsilon^{\kappa}$ is the multiplier associated with the final
constraint~\eqref{eq:innerequiv-cible} at iteration~$\kappa$. Several
cases have to be considered when solving Problem~\eqref{pb:innerequiv-problem}.
\begin{enumerate}
\item The sequence~$\ba{\delta^{\kappa}}_{\kappa\in\mathbb{N}}$
  does not stabilize: we consider that the algorithm diverges.
\item The sequence~$\ba{\delta^{\kappa}}_{\kappa\in\mathbb{N}}$
  converges to~$\delta\opt$:
  \begin{enumerate}
  \item if~$\nnorm{\delta\opt}=0$, the final target constraint
    is satisfied, the algorithm converges to a solution with
    a final state~$x\opt(\tf)$:
    \begin{itemize}
    \item if~$K\bp{x\opt(\tf)} - \mu \geq 0$, the solution
      of Problem~\eqref{pb:inner-problem} is to do nothing
      (see postponed Lemma~\ref{rm:negativity}),
    \item otherwise, the solution of Problem~\eqref{pb:inner-problem}
      is very similar to the solution of Problem~\eqref{pb:innerequiv-problem} since we are as close
      as possible to the target (see Appendix~\ref{sec:gradient}),
      and the optimal value~$\upsilon\opt$ of the multiplier allows
      to compute the derivatives of~$W_{r}$ used in the stochastic
      Arrow-Hurwicz algorithm,
    \end{itemize}
  \item otherwise, the target cannot be reached and~$\nnorm{\delta\opt}$
    is the distance by which the target is missed: the sequence of
    multipliers~$\ba{\upsilon^{\kappa}}_{\kappa\in\mathbb{N}}$ goes to infinity
    and the deterministic Arrow-Hurwicz algorithm aims to get as close as
    possible to the final target, with a final state~$x\opt(\tf)$:
    \begin{itemize}
    \item if~$K\bp{x\opt(\tf)} - \mu \geq 0$, the solution
      of Problem~\eqref{pb:inner-problem} is to do nothing,
    \item if $~\nnorm{\delta\opt} > r$, the solution
      of Problem~\eqref{pb:inner-problem} is to do nothing,
    \item otherwise, the solution of Problem~\eqref{pb:innerequiv-problem}
      is the solution of Problem~\eqref{pb:inner-problem}.
    \end{itemize}
  \end{enumerate}
\end{enumerate}
Note that the deterministic Arrow-Hurwicz algorithm used to solve
Problem~\eqref{pb:innerequiv-problem} takes the distinctions made in
Appendix~\ref{sec:gradient} concerning the study of
Problem~\eqref{pb:inner-problem} into account, that is, the case where the
target is exactly reached has to be treated separately of the case where the
target is reached up to~$r$. So far, we do not have formal proof of the
equivalence suggested above between Problems~\eqref{pb:inner-problem}
and~\eqref{pb:innerequiv-problem}.  This is an area where progress should be
made.

From the expression of the criterion~\eqref{eq:inner-crit} and
the shape of~$\fctI_{r}$ given by~\eqref{eq:smooth-indicator-function},
we deduce Lemma~\ref{rm:negativity} used in the heuristic description.
\begin{lemma}
  \label{rm:negativity}
  Let~$\bp{x\opt,v\opt}$ be an optimal solution of the inner
  problem~\eqref{pb:inner-problem} and assume that the optimal solution is
  such that $\bnorm{C\bp{x\opt(\tf)}} < r$ so that~$\fctI_{r}$
  is non zero. Then, we have that
  \begin{equation}
    K\bp{x\opt(\tf)} - \mu \leq 0  \eqfinp
  \end{equation}
\end{lemma}
\begin{proof}
  Indeed, \jpc{en fait c'est une hypothèse raisonable mais pas impliqué}
  it is always possible to miss the target and thus to
  obtain a cost equal to~$0$ thanks to the indicator function,
  so that~$K\bp{x\opt(\tf)}-\mu > 0$ cannot induce an optimal
  cost in~\eqref{pb:inner-problem}.
\end{proof}

\subsection{Parameters of the stochastic Arrow-Hurwicz
            algorithm\label{ssec:parameters}}

As indicated in~\S\ref{ssect:stochastic-arrow}, some parameters
have to be chosen when implementing the stochastic Arrow-Hurwicz
algorithm. As with any standard stochastic gradient algorithm, the
step lengths~$\epsilon_{u}^{k}$ and~$\epsilon_{\mu}^{k}$ must be
properly set. According to the theory (see~\cite{Culioli_CRAS_1995}),
these parameters should vary in $1/k$ and therefore be of
the general form:
\begin{equation}
\label{eq:step-lenght}
\frac{\alpha}{\beta+k} \eqfinp
\end{equation}

In order to properly set the smoothing coefficient~$r^{k}$,
we get inspiration from~\cite{Andrieu_EJOR_2011}
(with proofs in the companion paper~\cite{Andrieu_arXiv_2007}),
where a general problem of minimizing a criterion in expectation
under a probability constraint is studied. As a result,
the smoothing coefficient at
iteration~$k$ of the stochastic Arrow-Hurwicz algorithm
is chosen as (see details in Appendix~\ref{sec:analogy})
\begin{equation}
  \label{eq:smoothing-coeff}
  r^{k} = \frac{a}{b+k^{1/3}} \eqfinp
\end{equation}
Note that such a smoothing coefficient (in~$k^{-1/3}$) is divided by~$10$
after~$1000$ iterations, and by~$100$ after~$10^{6}$ iterations.

\section{Real-life application\label{sect:realproblem}}

In this section, we describe the specific problem we addressed
in this study and give numerical results for different probability
levels required for hitting the target.

\subsection{Problem data\label{ssec:problem-data}}

The numerical data~(see~\S\ref{ssect:model}) corresponding to the spatial
mission that we study in this paper can be found in~\cite{Dargent_Report_2004}.
They involve a rescaling for proper numerical treatment and are as follows.
\begin{itemize}
\item Parameter:\footnote{Before scaling,
the parameters~$T$, $g_{0}$ and~$I_{\mathrm{sp}}$ are~$0.2$,
$9.80665$ and~$1500$ respectively.}
\begin{equation*}
  T = 0.0336750 \eqsepv
  g_{0} I_{\mathrm{sp}} = 0.4936891 \eqsepv
  \paramu = 1 \eqfinp
\end{equation*}
\end{itemize}
\begin{itemize}
\item Initial and final times:\footnote{Before scaling,
the initial and final times are~$40$ and~$510$ days respectively.}
\begin{equation*}
  \ti = 0.6888699 \eqsepv \tf = 8.7830909 \eqfinp
\end{equation*}
\end{itemize}
\begin{itemize}
\item Satellite initial state:\footnote{Before scaling, the initial
mass~$x_{\mathrm{i,7}}$ of the satellite is~$1000$ kilograms.}
\begin{equation*}
  x_{\mathrm{i}} = \left(
    0.999702 \;\; - 0.003359 \;\; 0.016942 \;\;
    - 0.000011 \;\; 0.000007 \;\; 36.52939 \;\; 1
  \right)\transp \eqfinp
\end{equation*}
\end{itemize}
\begin{itemize}
\item Rendez-vous target coordinates:
\begin{multline*}
  \left( x_{\mathrm{f},1} \;\; x_{\mathrm{f},2} \;\; x_{\mathrm{f},3} \;\;
         x_{\mathrm{f},4} \;\; x_{\mathrm{f},5} \;\; x_{\mathrm{f},6} \right)
  = \\
  \left( 1.511514 \;\; 0.085367 \;\; - 0.037923 \;\;
         0.010474 \;\; 0.012275 \;\;  42.17610 \right) \eqfinp
\end{multline*}
\end{itemize}
\begin{itemize}
\item Failure probability laws \\
The instant at which the failure occurs and the failure duration
follow exponential probability laws, whose parameters are
\begin{itemize}
\item failure start time:
$\underline{t_{\mathrm{p}}}=0.68887$ and~$\lambda_{\mathrm{p}}=15.1711$,
\item failure duration:
$\underline{t_{\mathrm{d}}}=0.03444$ and~$\lambda_{\mathrm{d}}=0.05350$.
\end{itemize}
\end{itemize}
Then, the probability weight~$\pif$ of the no-failure scenario is:
\begin{equation}
\label{eq:pif}
\pif = 0.58653 \eqfinp
\end{equation}

\subsection{Deterministic problem solution\label{ssec:deter-solution}}

The convergence of the deterministic Arrow-Hurwicz algorithm as well as
the optimal control obtained are depicted Figure~\ref{fig:deter-problem}.

\begin{figure}[hbtp]
\begin{center}
  \mbox{\includegraphics[width=0.385\textwidth]{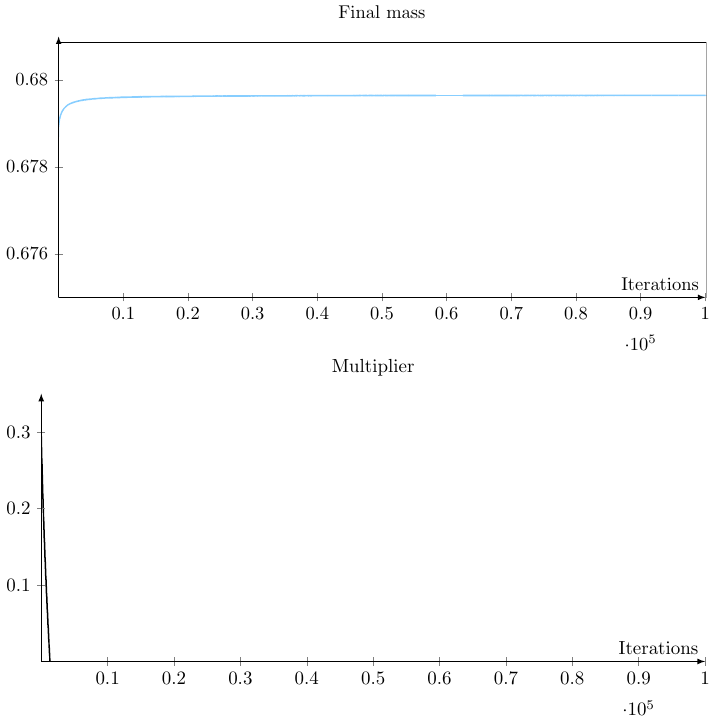}}
  \mbox{\includegraphics[width=0.5\textwidth]{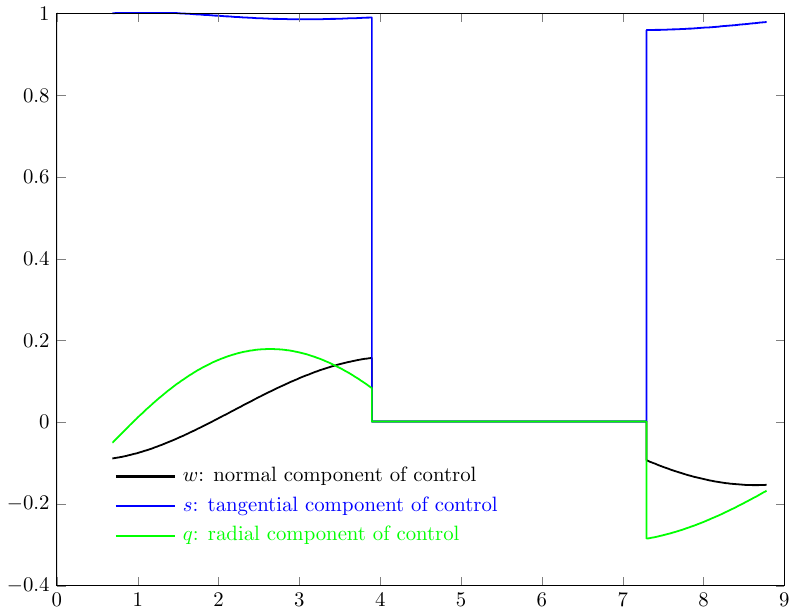}}
\end{center}
\caption{Convergence and solution of the deterministic problem}
\label{fig:deter-problem}
\end{figure}

In the deterministic solution, we distinguish three periods in
the mission duration:
\begin{itemize}
\item period~(a): engine at maximum thrust (between days~$40$ and~$226.3$),
\item period~(b): engine stopped (between days~$226.3$ and~$423.8$),
\item period~(c)~: engine at maximum thrust (between days~$423.8$ and~$510$).
\end{itemize}
After scaling, the values of the instants delimiting these three periods are:
\begin{itemize}
\item[] $\ti = 0.68887$ (initial time),
\item[] $t_{\mathrm{a}} \approx 3.8983$ (end of period~(a)),
\item[] $t_{\mathrm{b}} \approx 7.2980$ (end of period~(b)),
\item[] $\tf = 8.78309$ (final time).
\end{itemize}
These notations will be used when analyzing the results
of the stochastic Arrow-Hurwicz algorithm.

It is easy to prove that a failure occurring before the end
of period~(a) does nor prevent from reaching the target.
Furthermore, a failure occurring during period~(b) has no effect since
the engine is stopped. The probability weight~$\pi_{\mathrm{r}}$
that a failure occurs before the end of period~(b) is:
\begin{equation*}
\pi_{\mathrm{r}} \approx 0.35020 \eqfinp
\end{equation*}
Finally, the probability weight~$p^{\mathrm{det}}$ of reaching
the target when applying the deterministic optimal control is
\begin{equation}
\label{eq:pdet-value}
p^{\mathrm{det}} = \pif + \pi_{\mathrm{r}} \approx 0.93673 \eqfinp
\end{equation}

\subsection{Stochastic problem solutions for various probability values
\label{ssec:stoch-solution}}

We recall that each iteration of the stochastic Arrow-Hurwicz algorithm
mainly consists in solving
\begin{itemize}
\item the deterministic optimal control problem associated
with the inner problem,
\item the deterministic optimal control problem associated
with the projection.
\end{itemize}
On an Intel Core2 Quad CPU 2.83~Ghz computer running Linux, the CPU time
associated with the resolution of these two problems is around~$2.2$ seconds.
In all the numerical tests carried out in this study, the stochastic
Arrow-Hurwicz algorithm is initialized with the deterministic optimal
control trajectory, and an initial value of the multiplier~$\mu$ equal
to~$0.325$. The stochastic Arrow-Hurwicz algorithm is stopped after
$100,000$ iterations.

\subsubsection{Numerical tests with a low probability level~$p$
\label{sssec:low}}

To begin with, we are interested in probability levels~$p$ that
are strictly lower than the probability~$p^{\mathrm{det}}$
to reach the target when using the deterministic solution.
We observe (see Figure~\ref{fig:probasfaibles}) that the algorithm
converges very satisfactorily.

\begin{figure}[hbtp]
\begin{center}
  \mbox{
  \includegraphics[width=0.385\textwidth]{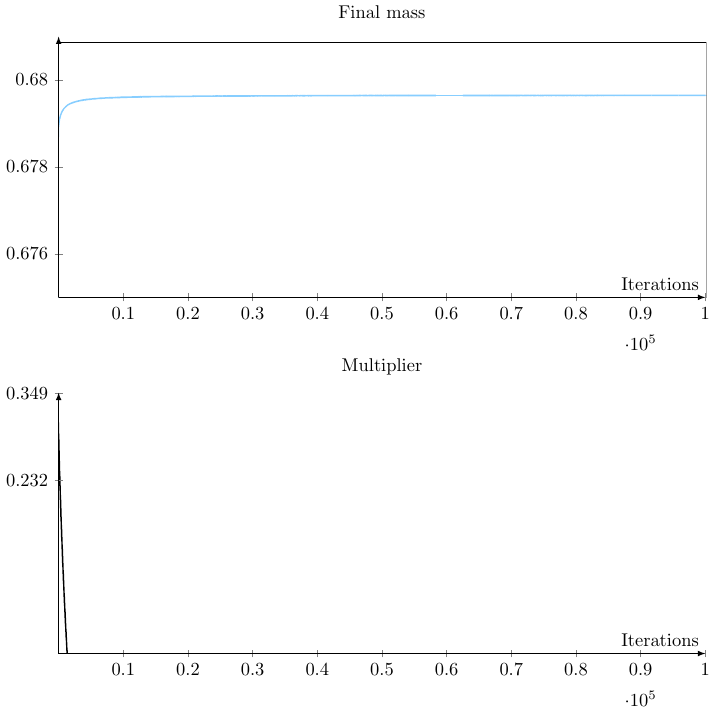}
  \includegraphics[width=0.5\textwidth]{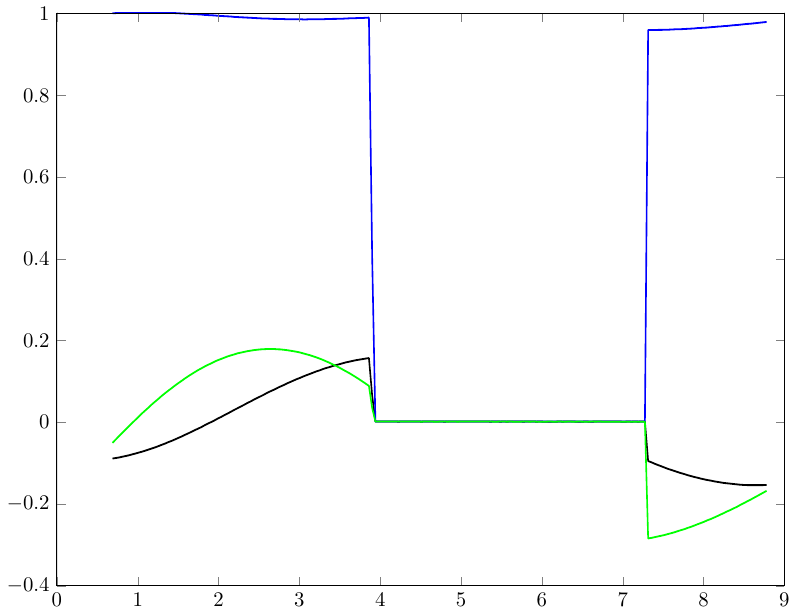}}

  \includegraphics[width=0.385\textwidth]{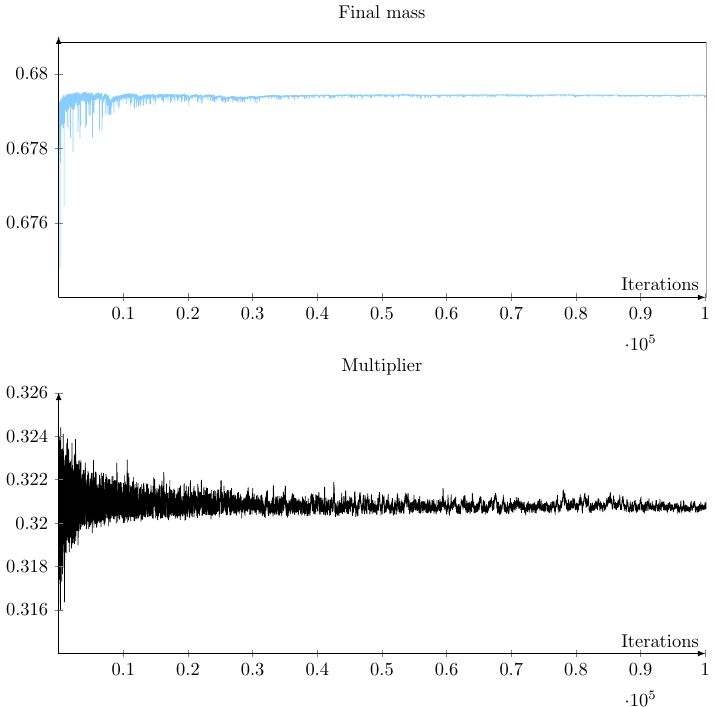}
  \includegraphics[width=0.5\textwidth]{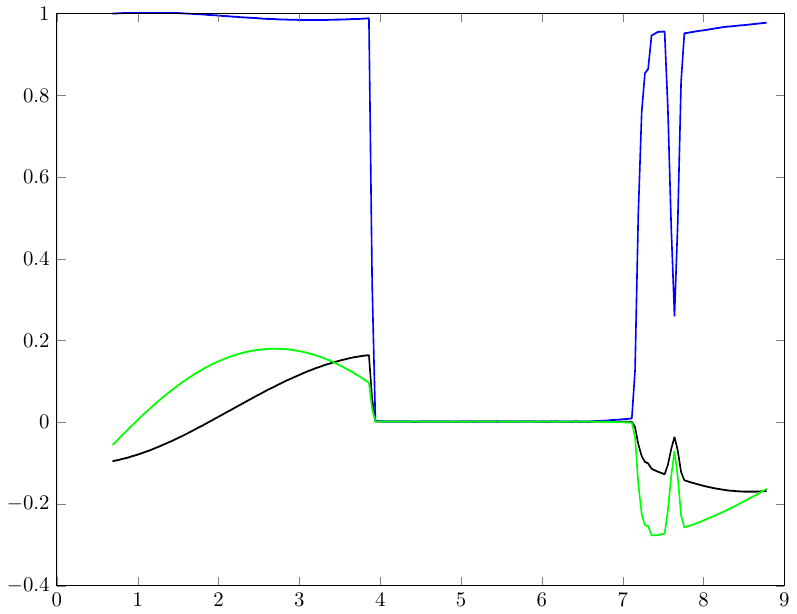}

  \includegraphics[width=0.385\textwidth]{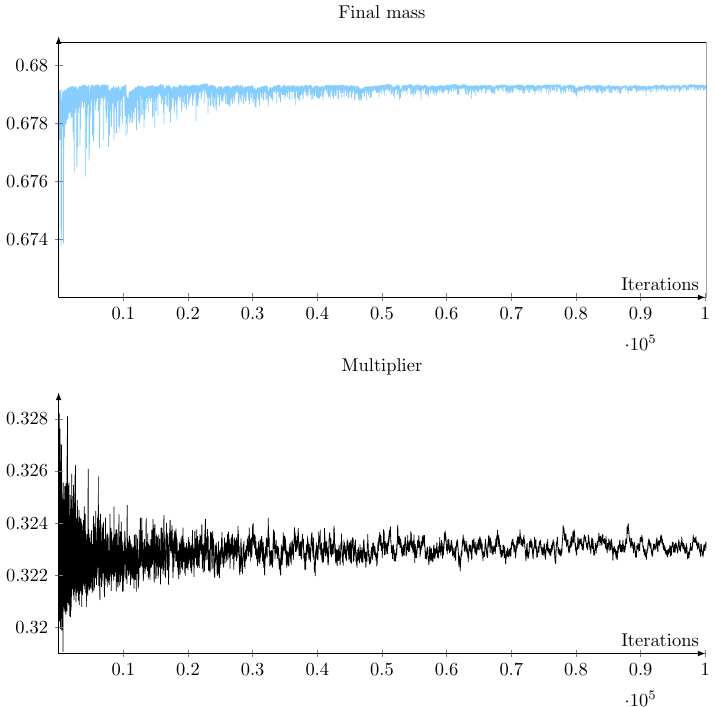}
  \includegraphics[width=0.5\textwidth]{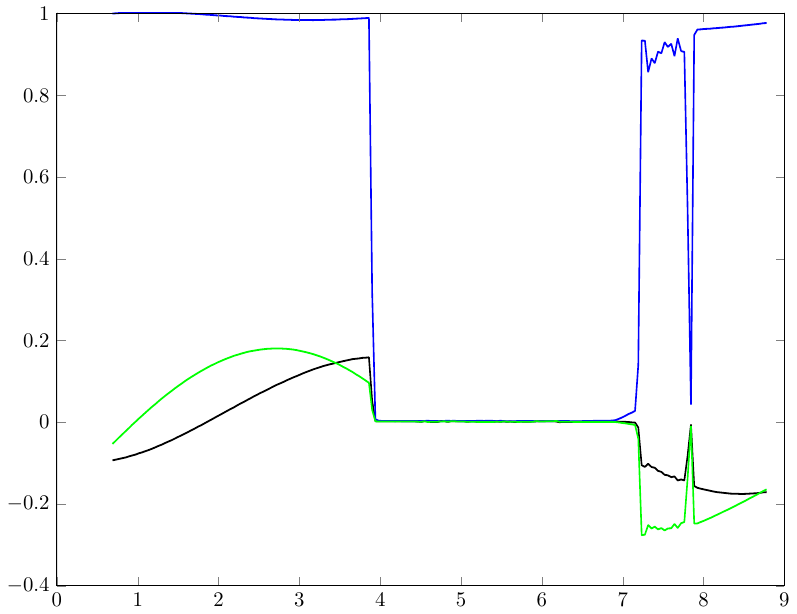}
\end{center}
\caption{Results for~$p=0.550$ (top), $p=0.750$ (middle) and~$p=0.925$ (bottom)}
\label{fig:probasfaibles}
\end{figure}

\begin{itemize}
\item For the probability level~$p=0.550$, that is, less than
the probability of the no-failure scenario~$\pif$ (see~\eqref{eq:pif}),
the control trajectories of the no-failure scenario are identical
to those of the deterministic optimal solution;
the probability constraint is not active (multiplier equal to~$0$
since~$\pif>0.550$) and the optimal fuel consumption is~$0.32024$.
\item For the probability level~$p=0.750$, the convergence is more
turbulent than in the previous case; the optimal
fuel consumption of the no-failure scenario is slightly higher than
in the previous case ($0.32045$), whereas the multiplier stabilizes
around the value~$0.3208$. However the control trajectories obtained
in this case are rather different from those associated with the no-failure
problem, the main difference being located after instant~$t_{\mathrm{b}}$
when a reduction of the engine speed appears.
\item For the probability level~$p=0.925$, the optimal fuel
consumption of the no-failure scenario is~$0.32054$, whereas
the multiplier stabilizes around the value~$0.3233$, that is, a small
difference with respect to the previous two cases.
The control distorsion after instant~$t_{\mathrm{b}}$ is more
pronounced than in the previous case, with very little
impact on the optimal consumption.
\end{itemize}
One might think that the optimal control for the problem under
probability constraint is identical to the optimal control
in the deterministic case as soon as the probability level~$p$
is less than~$p^{\mathrm{det}}$ in~\eqref{eq:pdet-value}.
In fact, this is not the case: the optimal control of
the deterministic problem is just \emph{admissible} for the
stochastic optimization problem, because the two problems
have different cost functions.

\subsubsection{Numerical tests with a medium probability level~$p$
\label{sssec:medium}}

We are then interested in probability levels~$p$ that
are slightly higher than the probability~$p^{\mathrm{det}}$
to reach the target when using the deterministic solution.
We observe (see Figure~\ref{fig:probasmoyennes}) that the algorithm
still converges very satisfactorily.

\begin{figure}[hbtp]
\begin{center}
  \includegraphics[scale=0.5]{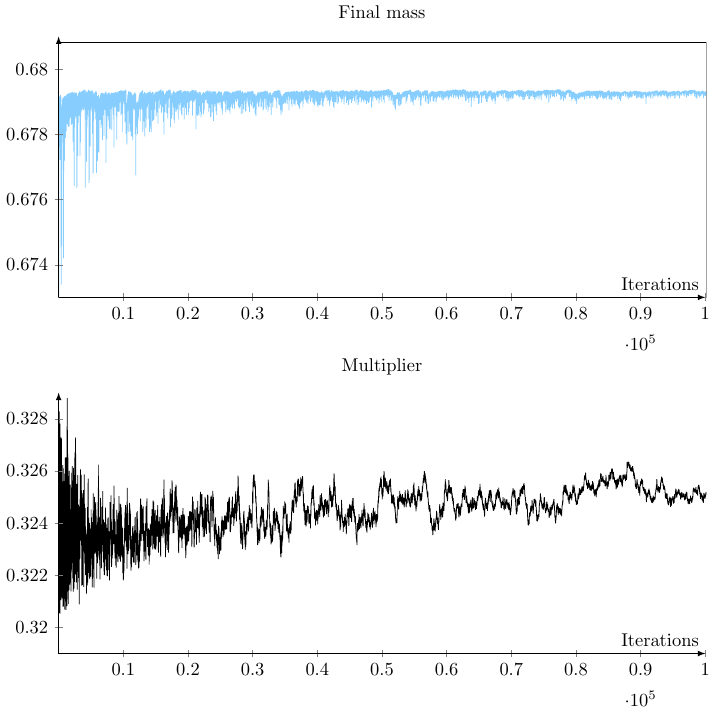}
  \includegraphics[scale=0.5]{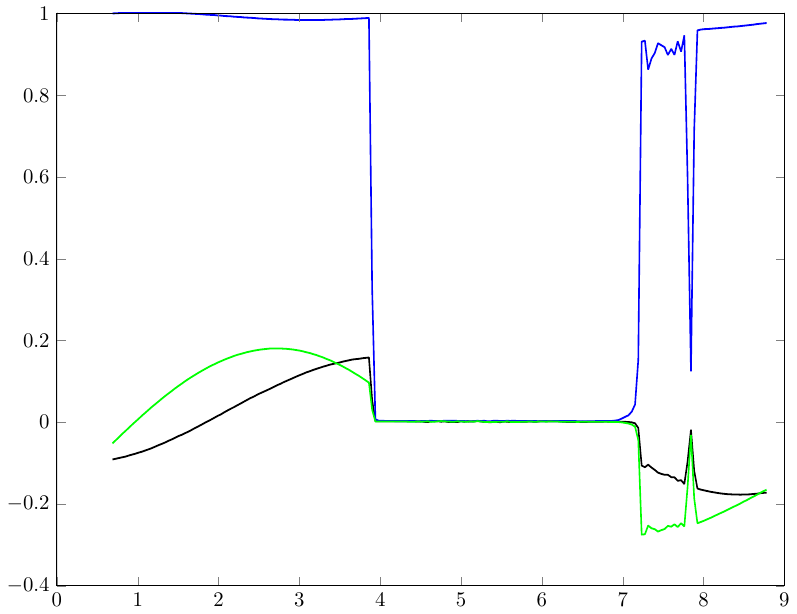}

  \includegraphics[scale=0.5]{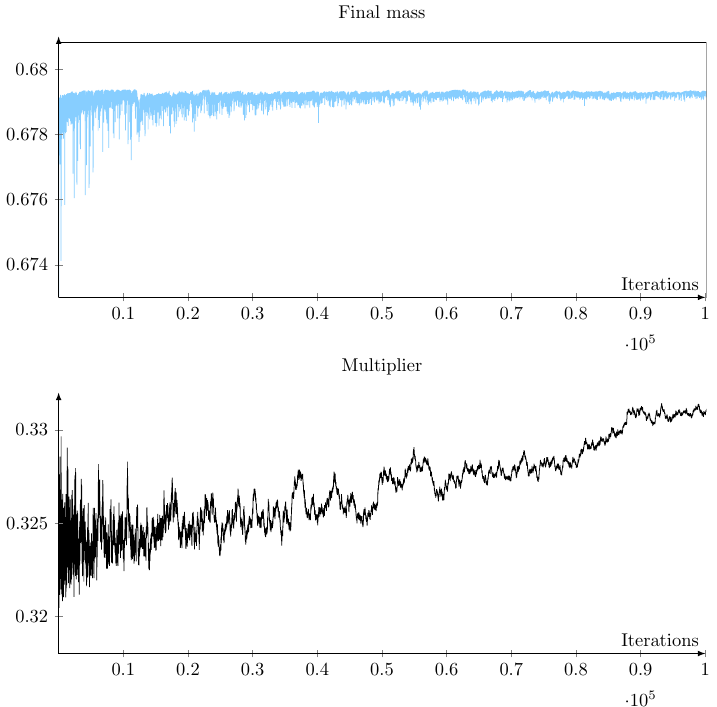}
  \includegraphics[scale=0.5]{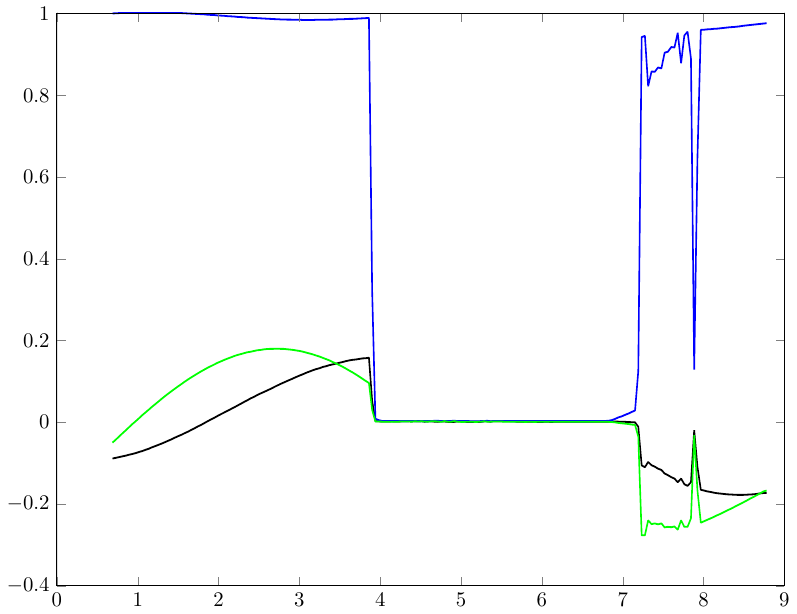}

  \includegraphics[scale=0.5]{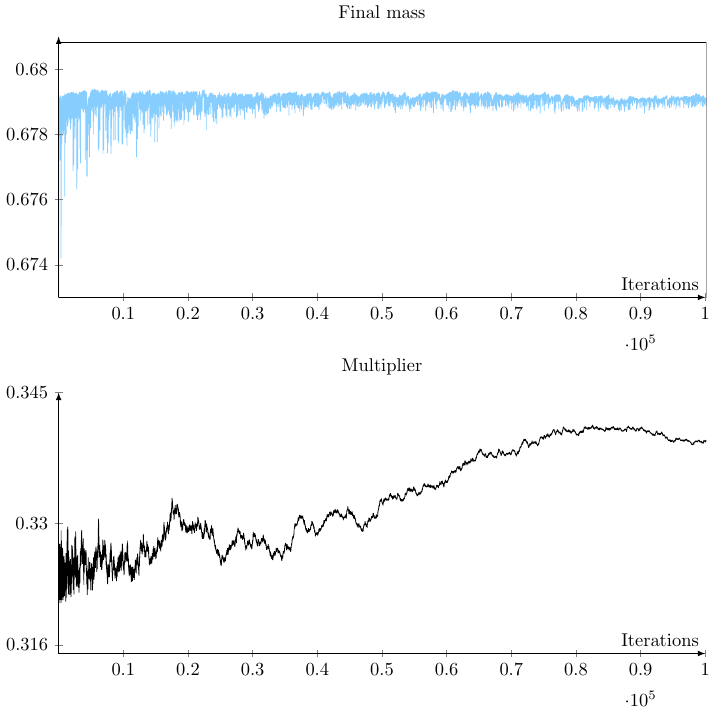}
  \includegraphics[scale=0.5]{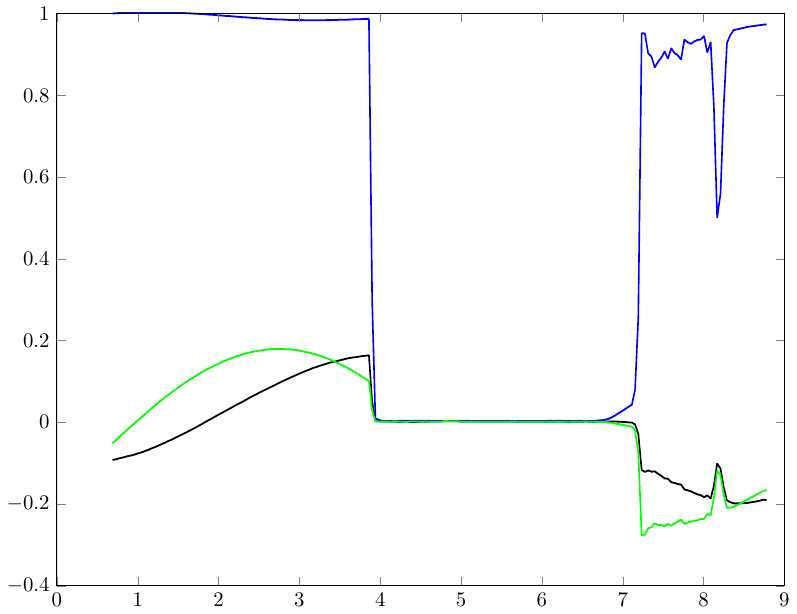}
\end{center}
\caption{Results for~$p=0.950$ (top), $p=0.960$ (middle) and~$p=0.970$ (bottom)}
\label{fig:probasmoyennes}
\end{figure}

\begin{itemize}
\item For the probability level~$p=0.950$, the solution behavior is very close
  to that observed in~\S\ref{sssec:low} for the probability level~$p=0.925$~:
  the optimal fuel consumption is~$0.3206$ and the multiplier stabilizes around
  the value~$0.3254$.
\item For a probability level~$p=0.960$, the optimal control of the no-failure
  scenario and the associated fuel consumption do not differ from those obtained
  for~$p=0.960$. But the optimal multiplier increases significantly, and its
  evolution during the algorithm is particularly slow during the first $40,000$
  iterations.
\item Finally, for the probability level~$p=0.970$, the multiplier increase over
  the iterations is even more impressive. The optimal control of the no-failure
  scenario also changes, the reduction of the engine speed after
  time~$t_{\mathrm{b}}$ occurring later than in previous cases.
\end{itemize}
For those medium values of the probability level, we can see
that satisfying the probability constraint is made possible
by two actions involved simultaneously.
\begin{enumerate}
\item On the one hand, the increase in the multiplier means that fewer and fewer
  failures are being rejected for consumption reasons (remember that the term
  $K\bp{x^{\xi}(\tf)}-\mu$ must remain negative for a failure to be recovered):
  during the last~$10,000$ iterations of the stochastic Arrow-Hurwicz algorithm,
  that is, when the algorithm is almost stabilized, the number of failures
  eliminated for this reason drops from~$250$ for~$p=0.950$ to~$6$
  for~$p=0.970$.
\item On the other hand, the later the engine speed is reduced during the last
  period~(c) of the mission, the more failures can be recovered, since any
  failure occurring between time~$t_{\mathrm{b}}$ and the time corresponding to
  this speed reduction time can be recovered.
\end{enumerate}

\subsubsection{Numerical tests with a high probability level~$p$
\label{sssec:high}}

We are now interested in probability levels~$p$ that
are significantly higher than the probability~$p^{\mathrm{det}}$
to reach the target when using the deterministic solution, namely,
$p=0.980$, $p=0.985$ and~$p=0.990$.
The results are presented at Figure~\ref{fig:probasfortes}.

\begin{figure}[hbtp]
\begin{center}
  \includegraphics[scale=0.5]{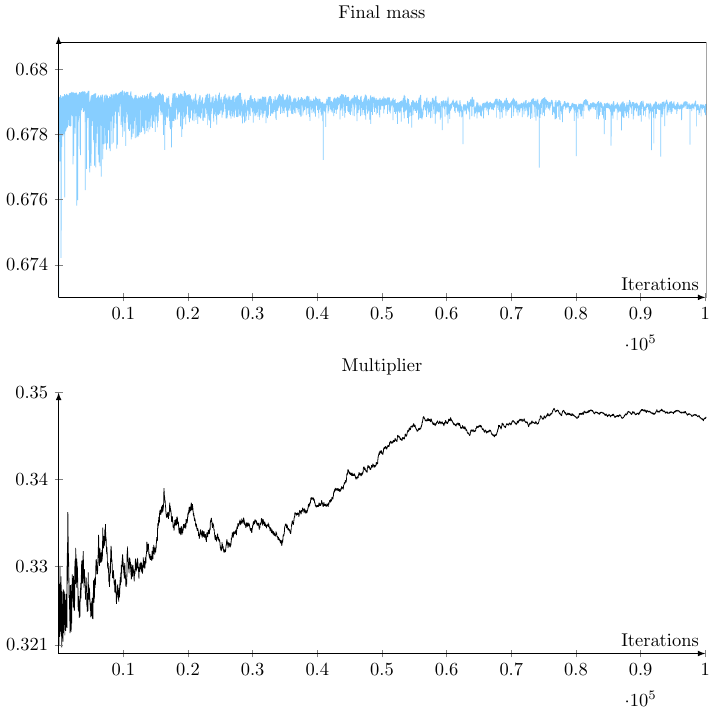}
  \includegraphics[scale=0.5]{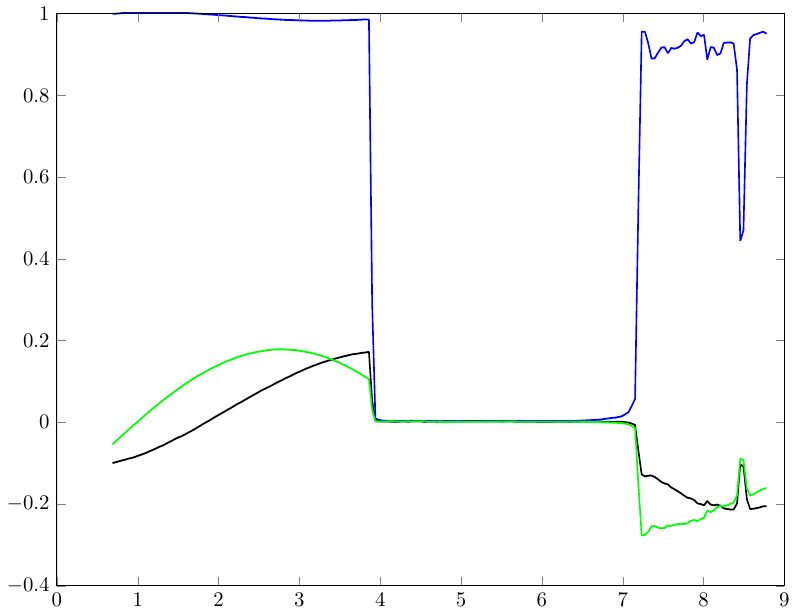}

  \includegraphics[scale=0.5]{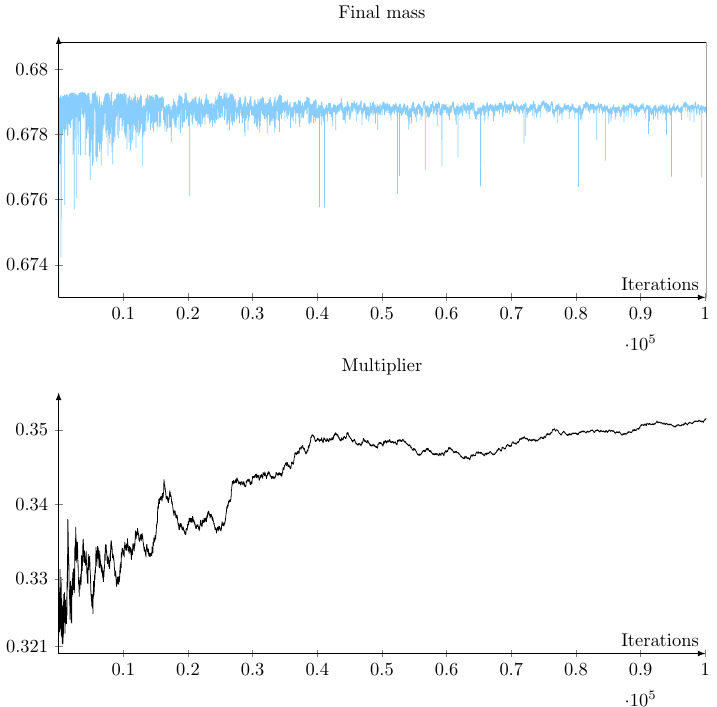}
  \includegraphics[scale=0.5]{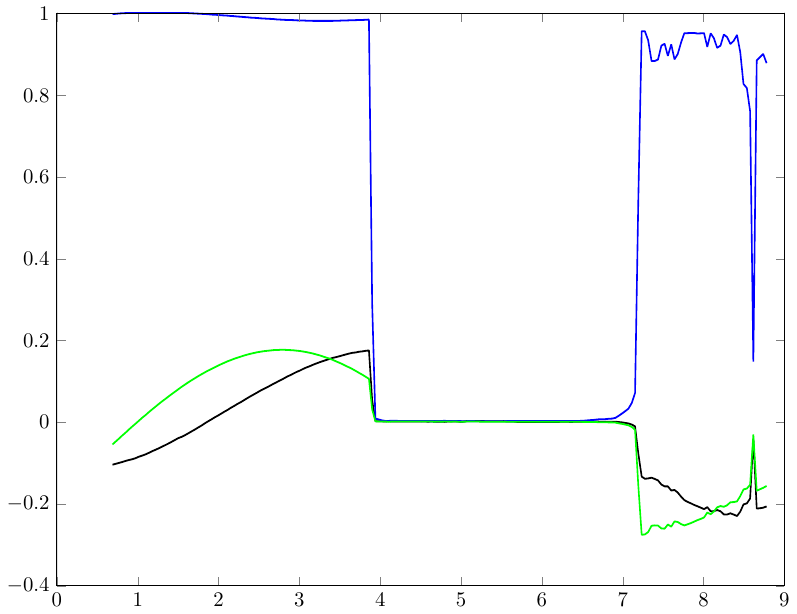}

  \includegraphics[scale=0.5]{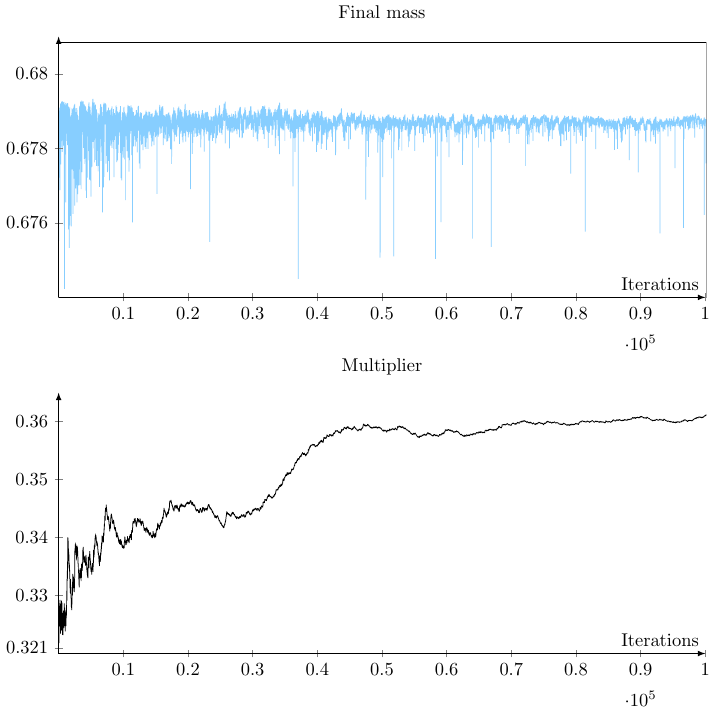}
  \includegraphics[scale=0.5]{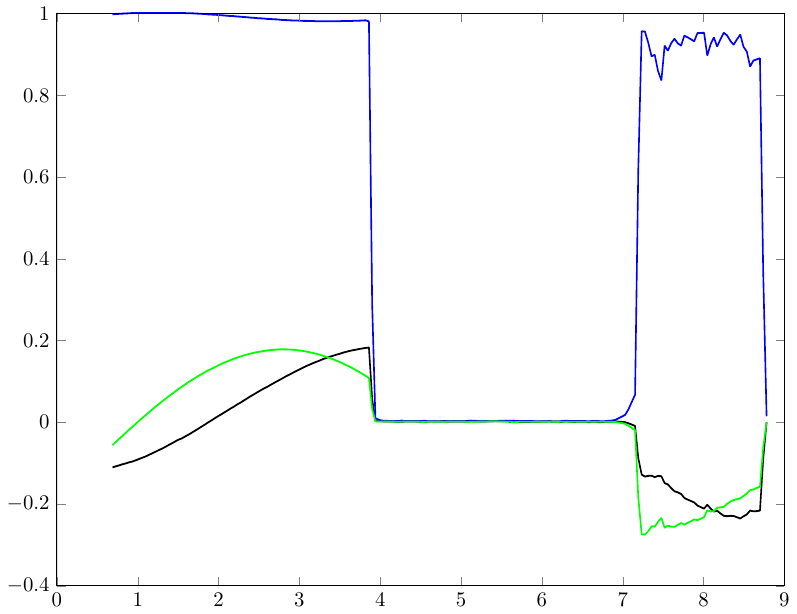}
\end{center}
\caption{Results for~$p=0.980$ (top),~$p=0.985$ (middle) and~$p=0.990$ (bottom)}
\label{fig:probasfortes}
\end{figure}

The results obtained for the medium probability levels are confirmed
by these new numerical experiments. Surprisingly, the convergence
of the multiplier~$\mu$ associated with the probability constraint
seems easier to achieve than in the case of medium probability levels.
As far as optimal control for the no-failure scenario is concerned,
delaying the decrease in engine speed more and more as~$p$ increases
(until the final instant~$\tf$ is reached) is clearly highlighted.

So it seems that the overall algorithm works reasonably well, but
several improvements would still be desirable. These include the
fact that a few tens of thousands of iterations at the start of the
algorithm do not seem to provide much improvement, and it would be
worth digging a little deeper to see if they could be dispensed with.
If we add to this the fact that the computation time of an iteration
is a little over~$2$ seconds for solving both the inner problem
and the projection problem, and therefore that an execution of the
algorithm requires a CPU time of the order of two and a half days,
we understand that there is also a stake in finding a more efficient
method of solving, for example shooting methods \cite{Bonnans_IFAC_2013}.

\subsubsection{Numerical tests with an extreme probability level~$p$
\label{sssec:extreme}}

Finally, we are interested in ``extreme'' probability levels~$p$,
that is, allowing to reach the target for~$99.5\%$ and~$99.9\%$ of
failures. The results are presented at Figure~\ref{fig:probasextremes}.

\begin{figure}[hbtp]
\begin{center}
  \includegraphics[scale=0.5]{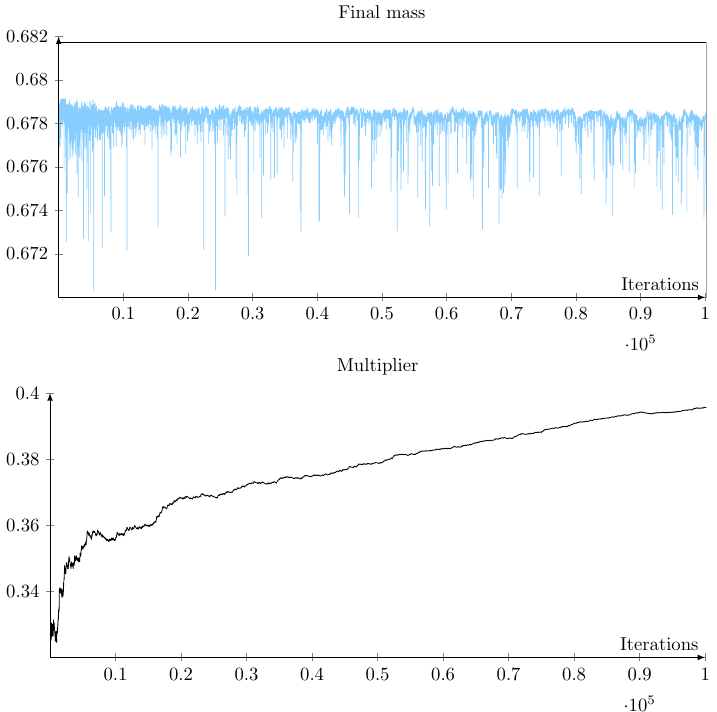}
  \includegraphics[scale=0.5]{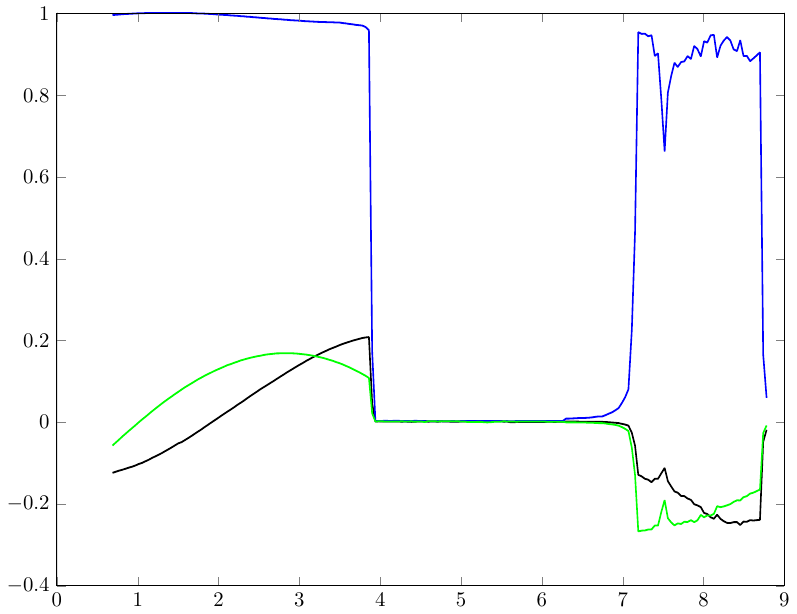}

  \includegraphics[scale=0.5]{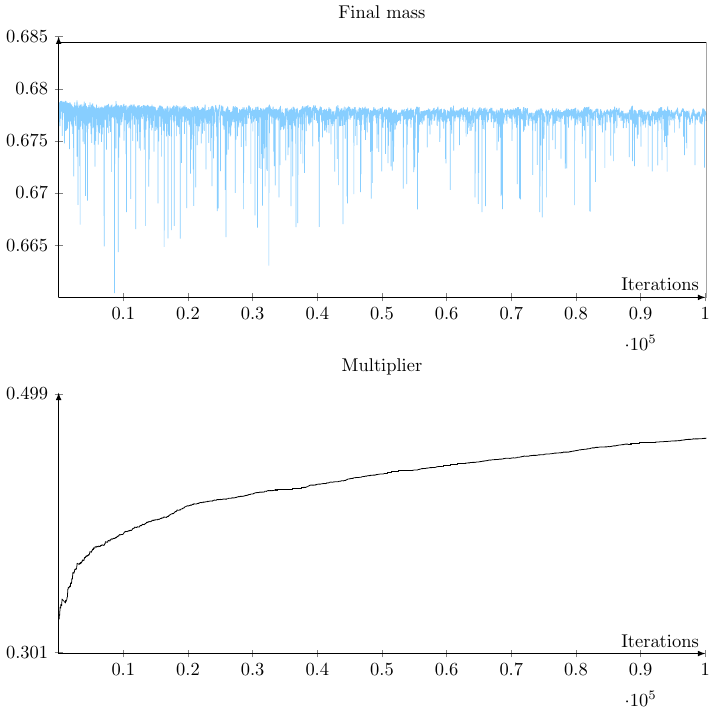}
  \includegraphics[scale=0.5]{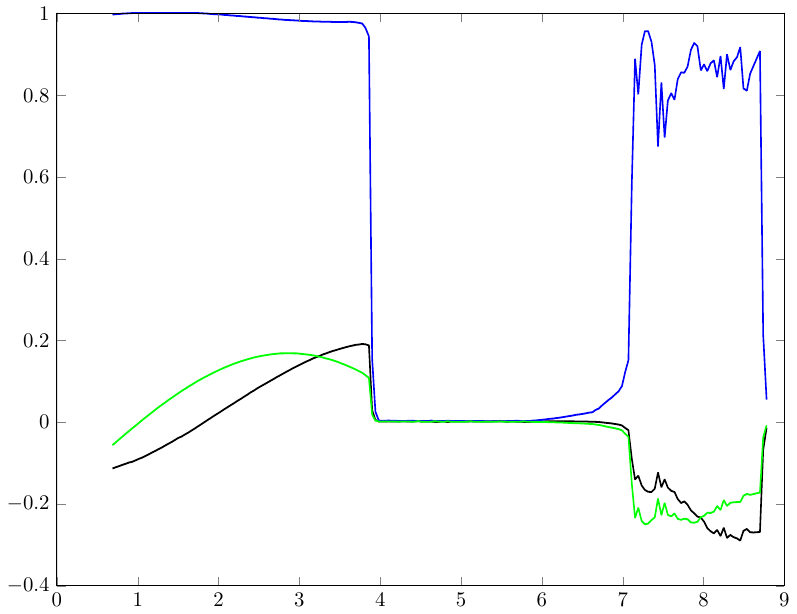}
\end{center}
\caption{Results for~$p=0.995$ (top) and~$p=0.999$ (bottom)}
\label{fig:probasextremes}
\end{figure}

The results show that the dual variable~$\mu$ has not converged and
continues to grow after~$100,000$ iterations. This is confirmed by
the fact that the real probability of recovering a failure at the end
of each run of the algorithm (respectively~$99.3\%$ and~$99.6\%$) are
a little lower than the probability levels requested.
In the case of~$p=0.999$, we might even ask whether such a probability
is feasible. In fact, it is not yet possible to determine the probability
level above which the problem can no longer be solved. If we look at
the shape of the optimal control along the no-failure scenario,
we see the following new phenomenon: the restart of the engine
after period~(b) occurs earlier than in the case of high
probability level in~\S\ref{sssec:high}, this renewed activity
before~$t_{\mathrm{b}}$ being compensated by a reduction of the
engine speed a little after~$t_{\mathrm{b}}$
(as before, there is a sharp speed reduction just before the final
instant~$\tf$). This is most probably how the algorithm tries to achieve
the required probability level, but this is still to be confirmed.

\subsection{Result summary\label{ssec:result-summary}}

To conclude, as a function of the requested level of probability~$p$,
we present two graphics showing the variation
of the optimal multiplier associated with the constraint
on the one hand, and the variation of the optimal consumption
of the no-failure trajectory on the other hand (each point on
these curves corresponds to a run of the stochastic Arrow-Hurwicz
algorithm).

The first curve (Figure~\ref{fig:multiplicateur-proba}) illustrates
the consistency of the results obtained. We can see that the multiplier
(which corresponds to the sensitivity of the optimal cost to the probability
level~$p$) increases with~$p$, with a break in the slope at a probability
level of around~$0.960$.

\begin{figure}[hbtp]
  \begin{center}
    \includegraphics[width=0.5\textwidth]{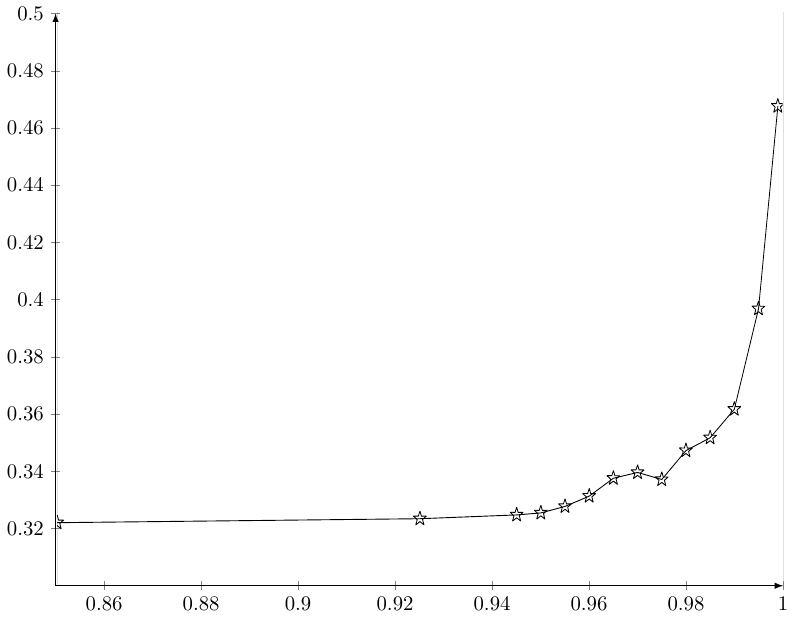}
  \end{center}
  \caption{Multiplier~$\mu$ as a function of the probability level~$p$}
  \label{fig:multiplicateur-proba}
\end{figure}

The second curve (Figure~\ref{fig:consommation-proba}) has a more direct
interpretation, since it indicates the amount of fuel needed to reach the end of
the no-failure scenario, according to the requested level of probability of
success of the mission. Once again, we see a break at around~$p=0.960$, as
consumption starts to increase much more rapidly after this value.  We can
therefore deduce that there is a probability threshold above which security
``comes at a high price''. Note that the fuel consumption for the no-failure
scenario shown in Figure~\ref{fig:consommation-proba} does not correspond to the
amount of fuel you would need to carry to complete the mission with a given
level of probability.  Indeed, this amount of fuel doesn't take into account the
recourse control consumption, which could be numerically estimated by Monte
Carlo approach. This calculation was not carried out in this study.

\begin{figure}[hbtp]
\begin{center}
  \includegraphics[width=0.5\textwidth]{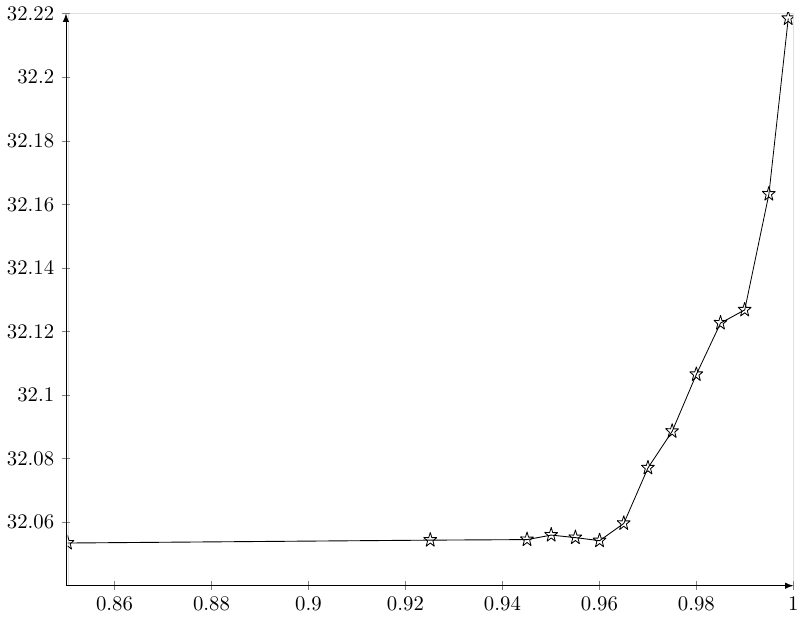}
\end{center}
\caption{Fuel consumption (multiplied by $10^2$) as a function of the probability level~$p$}
\label{fig:consommation-proba}
\end{figure}

\section{Conclusions and perspectives\label{sect:conclusion}}

In this study dedicated to the design of ``robust'' interplanetary
trajectories, we believe that we have shown that the method that
consists in solving the problem of minimizing fuel consumption
in expectation under the probability constraint of reaching the final
target by the stochastic gradient method works effectively.
This involved overcoming a number of difficulties, such as choosing
a suitable problem parametrization, dealing with the final target
constraint of the no-failure scenario by projection, among others.
The results are sufficiently clear, at least for probability levels~$p$
up to~$0.990$, to highlight the modifications of the optimal control
as a function of~$p$. Moreover, the convergence of the stochastic algorithm,
although slow (but this is largely due to the way in which the coefficient
for smoothing the constraint in probability must decrease), is reasonable
in the majority of cases studied. However, much remains to be done,
both theoretically and numerically.
\begin{itemize}
\item As far as theory is concerned, and as suggested
in Appendices~\ref{sec:analogy} and~\ref{sec:gradient},
we need to understand how
the inner problem can be formulated with the Heaviside function~$Y$
(for which a theory of smoothing exists: see~\cite{Andrieu_EJOR_2011}).
This question is all the more important in practical terms as
the smoothing of the function~$Y$ can be performed symmetrically
with respect to the point of discontinuity, and this leads to
an optimal rule of decreasing smoothing coefficient in~$k^{-1/5}$ and
decreasing mean square error in~$k^{-4/5}$ (division by 10,000
in 100,000 iterations. In the case of non-symmetrical smoothing,
we move to an optimal rule in~$k^{-1/3}$ and a mean-square error
decreasing in~$k^{-2/3}$ (division by about 2,000 in 100,000 iterations).
By Appendix~\ref{sec:analogy}, we are convinced that
direct smoothing of~$\fctI$ is equivalent to a non-symmetrical
smoothing of~$Y$, that is, not the best rule to decrease the mean
square error.
\item From the numerical point of view, the overall calculation time
is enormous, making the operational use of the program still too cumbersome.
Part of the explanation lies in the number of iterations required, as
explained above. The other part of the explanation comes from the cost
of solving a single iteration, which requires the complete resolution
of two deterministic optimal control problems. Any progress on either
of these two operations, and particularly the first, would have
a beneficial effect on this overall computation time.
\end{itemize}

\appendix

\section{Satellite model details\label{sec:satellite-details}}

\subsection{Change of coordinates}

We recall here the simplified dynamic model of a satellite given in Equation~\ref{eq:dynamics}
\begin{align*}
  \frac{\mathrm{d}r}{\mathrm{d}t} = v
  \eqsepv \quad
  \frac{\mathrm{d}v}{\mathrm{d}t} = - \paramu \frac{r}{\norm{r}^{3}} +
  \frac{T}{m}\kappa
  \eqsepv
  \frac{\mathrm{d}m}{\mathrm{d}t} = - \frac{T}{g_{0}I_{\mathrm{sp}}} \delta
  \eqfinv
\end{align*}
and rewrite this state dynamics using the equinoctial
coordinates~$(p,e_{x},e_{y},h_{x},h_{y},l)$, for the position and velocity
of the satellite. The equinoctial coordinates are obtained from the Keplerian
coordinates\footnote{semi-major axis, eccentricity, inclination, argument
of periapsis, longitude of the ascending node, and true anomaly}
$(a,e,i,\omega,\Omega,\nu)$ through the transformation:
\begin{align*}
  p     & = a | 1-e^{2}| \; , \\
  e_{x} & = e \cos(\omega+\Omega) \; , \\
  e_{y} & = e \sin(\omega+\Omega) \; , \\
  h_{x} & = \tan (i/2) \cos(\Omega) \; , \\
  h_{y} & = \tan (i/2) \sin(\Omega) \; , \\
  l     & = \omega + \Omega + \nu \; ,
\end{align*}
and by performing the following change of variables for the control
inputs:\footnote{where~$\alpha$ is the azimuth measured relative to~$q$
and~$\beta$, the elevation of the engine thrust in the satellite's
local frame of reference}
\begin{align*}
  q & = \delta \cos(\beta) \cos(\alpha) \; , \\
  s & = \delta \cos(\beta) \sin(\alpha) \; , \\
  w & = \delta \sin(\beta) \; ,
\end{align*}
where~$q$, $s$ and~$w$ are the radial, tangential, and normal components
of the control input, respectively.
% \begin{figure}[hbtp]
%   \begin{center}
%     \includegraphics[scale=0.75]{coordonnees}
%   \end{center}
%   \caption{Coordinate system}
%   \label{coordonnees}
% \end{figure}
In this coordinate system, Gauss' equations are written as:
\begin{subequations}
  \label{eq:dynamique}
  \begin{align}
    \frac{\mathrm{d}p}{\mathrm{d}t}
    & = 2 \: \sqrt{\frac{p^{3}}{\paramu}} \: \frac{1}{Z} \:
      \frac{T}{m} \: s \; , \\
    \frac{\mathrm{d}e_{x}}{\mathrm{d}t}
    & = \sqrt{\frac{p}{\paramu}} \: \frac{1}{Z} \:
      \frac{T}{m} \big( Z\:\sin(l)\:q + A\:s - e_{y}\:F\:w \big) \; , \\
    \frac{\mathrm{d}e_{y}}{\mathrm{d}t}
    & = \sqrt{\frac{p}{\paramu}} \: \frac{1}{Z} \:
      \frac{T}{m} \big( -Z\:\cos(l)\:q + B\:s + e_{x}\:F\:w \big) \; , \\
    \frac{\mathrm{d}h_{x}}{\mathrm{d}t}
    & = \frac{1}{2} \: \sqrt{\frac{p}{\paramu}} \: \frac{X}{Z} \: \cos(l) \:
      \frac{T}{m} \: w \; , \\
    \frac{\mathrm{d}h_{y}}{\mathrm{d}t}
    & = \frac{1}{2} \: \sqrt{\frac{p}{\paramu}} \: \frac{X}{Z} \: \sin(l) \:
      \frac{T}{m} \: w \; , \\
    \frac{\mathrm{d}l}{\mathrm{d}t}
    & = \sqrt{\frac{\paramu}{p^{3}}} \: Z^{2} +
      \sqrt{\frac{p}{\paramu}} \: \frac{1}{Z} \: F \: \frac{T}{m} \: w \; , \\
    \frac{\mathrm{d}m}{\mathrm{d}t}
    & = -\frac{T}{g_{0}I_{\mathrm{sp}}} \: \sqrt{q^{2}+s^{2}+w^{2}}\; ,
  \end{align}
\end{subequations}
with~:
\begin{align*}
  Z & = 1 + e_{x}\cos(l) + e_{y}\sin(l) \; , \\
  A & = e_{x} + (1+Z)\cos(l) \; , \\
  B & = e_{y} + (1+Z)\sin(l) \; , \\
  F & = h_{x}\sin(l) - h_{y}\cos(l) \; , \\
  X & = 1 + h_{x}^{2} + h_{y}^{2} \; .
\end{align*}
To apply variational techniques, we assume that the engine on/off
indicator can vary continuously between the values~$0$ (off) and~$1$ (on),
so that the constraint on the control input is written as:
\begin{equation}
  \label{eq:commande}
  q^{2}+s^{2}+w^{2} \leq 1 \; .
\end{equation}

Finally, the initial time~$\ti$ is assumed to be fixed, and the state
at the initial time is assumed to be known giving $x(\ti) = x_{\mathrm{i}}$.
Thus, with~$x=(p,e_{x},e_{y},h_{x},h_{y},l)^{\top}$ and~$u=(q,s,w)^{\top}$,
the equinoctial state dynamics~\eqref{eq:dynamique} can be summarized as
\begin{equation}
  x(\ti) = x_{\mathrm{i}} \eqsepv
  \frac{\mathrm{d}x}{\mathrm{d}t}(t) = f\bp{x(t),u(t)}
  \eqfinv
\end{equation}
which gives Equation~\eqref{eq:equinoxial}.

\subsection{Deterministic optimization problem}

We recall that the final time~$\tf$ is also fixed. The optimization
problem we aim to solve consists of enforcing a rendez-vous constraint
in position/velocity (which thus affects the first six components
of the state), with the objective of maximizing the mass
at the final time~$\tf$.

Let~$\lambda=(\lambda_{p},\lambda_{e_{x}},\lambda_{e_{y}}, \lambda_{h_{x}},\lambda_{h_{y}},\lambda_{l}, \lambda_{m})^{\top}$ be
the adjoint state associated with the dynamics~\eqref{eq:dynamique},
the Hamiltonian of the problem is expressed as (with no integral term
in the cost):
\begin{equation}
  \label{eq:hamiltonien}
  H(x,u,\lambda) = \lambda^{\top} f(x,u) \; ,
\end{equation}
The dynamics of the adjoint state~$\lambda$ are given by:\footnote{The
explicit form of these equations is relatively complicated, and therefore
it will not be provided in this note. The reader is referred
to~\cite{Dargent_Report_2004} for the complete expression}
\begin{equation}
  \label{eq:etatadjoint}
  \frac{\mathrm{d}\lambda}{\mathrm{d}t} = -
  \Big( \frac{\partial f}{\partial x}(x,u)\Big)^{\top} \lambda \; .
\end{equation}
The gradients of the Hamiltonian with respect to the control
inputs are written as:
\begin{equation}
  \label{eq:gradient-u}
  \Big( \frac{\partial f}{\partial u}(x,u)\Big)^{\top} \lambda \; ,
\end{equation}
with~:
\begin{equation*}
  \frac{\partial f}{\partial u}(x,u) =
  \left(
    \begin{array}{ccc}
      0                 & 2 K_{a} p & 0                           \\
      K_{a} Z \sin(l) & K_{a} A      & - K_{a} e_{y} F             \\
      - K_{a} Z \cos(l) & K_{a} B      &   K_{a} e_{x} F             \\
      0                 & 0            & \frac{1}{2} K_{a} X \cos(l) \\
      0                 & 0            & \frac{1}{2} K_{a} X \sin(l) \\
      0                 & 0            & K_{a} F                     \\
      \gamma q_{n}      &\gamma s_{n}  &\gamma w_{n}
    \end{array}
  \right)
  \eqfinv
\end{equation*}
where~$q_{n}$, $s_{n}$ and~$w_{n}$ are the normalized control inputs,
that is~$q_{n}=\frac{q}{\sqrt{q^{2}+s^{2}+w^{2}}}$ , \ldots, and where
the two variables~$\gamma$ and~$K_{a}$ are defined as:
\begin{align*}
  \gamma & = - \frac{T}{g_{0}I_{\mathrm{sp}}} \; , \\
  K_{a}  & = \sqrt{\frac{p}{\paramu}} \frac{T}{m} \frac{1}{Z} \; .
\end{align*}
The final cost to minimize consists, on one hand, of the term corresponding
to the negative of the mass at the final time, and on the other hand,
of a term corresponding to the dualization of the final target
constraints\footnote{Recall that our goal is to use a stochastic algorithm
to solve the optimization problem under probabilistic constraints.
Therefore, we are compelled to dualize the final target constraints}
using an augmented Lagrangian with a regularization parameter~$c$,
the multiplier being denoted as~$\mu^{\mathrm{d}}$, thus in total:
\begin{equation}
  \label{eq:cout-final}
  x_{7}(\ti)-x_{7}(\tf) \; + \; \sum_{j=1}^{6}
  \mu^{\mathrm{d}}_{j} \big( x_{j}(\tf)-x_{\mathrm{f},j} \big) \; + \;
  \frac{c}{2} \big( x_{j}(f)-x_{\mathrm{f},j} \big)^{2} \; .
\end{equation}
The transversality conditions then provide the value of the adjoint
state at the final time:
\begin{equation}
  \label{eq:transver-m}
  \lambda(\tf) =
  \left(
    \begin{array}{c}
      \mu^{\mathrm{d}}_{1}+c\big(x_{1}(\tf)-x_{\mathrm{f},1}\big) \\
      \mu^{\mathrm{d}}_{2}+c\big(x_{2}(\tf)-x_{\mathrm{f},2}\big) \\
      \mu^{\mathrm{d}}_{3}+c\big(x_{3}(\tf)-x_{\mathrm{f},3}\big) \\
      \mu^{\mathrm{d}}_{4}+c\big(x_{4}(\tf)-x_{\mathrm{f},4}\big) \\
      \mu^{\mathrm{d}}_{5}+c\big(x_{5}(\tf)-x_{\mathrm{f},5}\big) \\
      \mu^{\mathrm{d}}_{6}+c\big(x_{6}(\tf)-x_{\mathrm{f},6}\big) \\
      -1
    \end{array}
  \right)
  \eqfinp
\end{equation}

The deterministic optimization problem is solved by the deterministic
Arrow-Hurwicz algorithm. An iteration of this algorithm consists of:
\begin{itemize}
\item Integrating the state dynamics~\eqref{eq:dynamique} over the
  interval~$[\ti,\tf]$ using the current control input~$u$.
\item Integrating the adjoint state dynamics~\eqref{eq:etatadjoint}
  from the conditions~\eqref{eq:transver-m} backward in time.
\item Updating the control input trajectories~$u$ using the
  gradients~\eqref{eq:gradient-u}, with projection onto the set defined
  by~\eqref{eq:commande}.
\item Performing a gradient step on the dual variables~$(\mu^{\mathrm{d}}_{1},\ldots,\mu^{\mathrm{d}}_{6})$.
\end{itemize}

\section{Transformation lemma\label{sec:form-lemma}}

Consider the following optimization problem:
\begin{equation}
  \label{pb:compacti}
  \min_{\mathbf{u} \in U\ad \subset\mathcal{U}} \;\;
  \frac{J(\mathbf{u})}{\Theta(\mathbf{u})} \quad
  \text{s.t.}\quad \Theta(\mathbf{u})\geq p \eqsepv
\end{equation}
in which $J$ and $\Theta$ assume positive values.
Problem~\eqref{pb:compacti} displays a very special structure,
in the sense that the denominator in the ratio defining the cost
function is identical to the left-hand side of the constraint.

The following lemma provides conditions under which
Problem~\eqref{pb:compacti} and the optimization problem
\begin{equation}
  \label{pb:compact}
  \min_{\mathbf{u} \in U\ad \subset\mathcal{U}} \;\;
  J(\mathbf{u}) \quad \text{s.t.} \quad \Theta(\mathbf{u})\geq p \eqfinv
\end{equation}
have the same solutions.
\medskip

\begin{lemma}
  Let~$U\ad$ be a closed convex subset of the Hilbert space~$\mathcal{U}$.
  Let~$J$ and~$\Theta$ be real-valued differentiable functions defined
  on~$\mathcal{U}$. We moreover assume that~$J$ is nonnegative
  and that~$\Theta$ is positive.
  \begin{itemize}
  \item Let~$\mathbf{u}\opt$ be a solution of~\eqref{pb:compacti}.
    If~$\mathbf{u}\opt$ satisfies the condition
    \begin{equation*}
      \Theta(\mathbf{u}\opt)=p \eqsepv
    \end{equation*}
    then $\mathbf{u}\opt$ is also a solution of~\eqref{pb:compact}.
  \item Conversely, suppose that $\mathbf{u}\opt$ satisfies
    the Karush-Kuhn-Tucker conditions of Problem~\eqref{pb:compact},
    and let~$\mu\opt$ be the associated multiplier.
    If~$\mu\opt$ is such that
    \begin{equation*}
      \mu\opt\geq \frac{J(\mathbf{u}\opt)}{\Theta(\mathbf{u}\opt)} \eqfinv
    \end{equation*}
    then $\mathbf{u}\opt$ also satisfies the Karush-Kuhn-Tucker conditions
    of Problem~\eqref{pb:compacti}.
  \end{itemize}
  \label{lem:useful}
\end{lemma}

\begin{proof}
  Assume that~$\mathbf{u}\opt$ is a solution of~\eqref{pb:compacti}
  such that~$\Theta(\mathbf{u}\opt)=p$. Then
  \begin{equation*}
    \frac{J(\mathbf{u}\opt)}{\Theta(\mathbf{u}\opt)} \leq
    \frac{J(\mathbf{u})}{\Theta(\mathbf{u})} \quad
    \forall \mathbf{u} \in U\ad \cap
    \{\mathbf{u}\in\mathcal{U} , \; \Theta(\mathbf{u}) \geq p\} \eqfinp
  \end{equation*}
  The left-hand side of the last inequality is equal to~$J(\mathbf{u}\opt)/p$,
  whereas the right-hand side is less than or equal to~$J(\mathbf{u})/p$.
  Hence~$\mathbf{u}\opt$ is a solution of Problem~\eqref{pb:compact}.

  Conversely, let~$(\mathbf{u},\opt\mu\opt)$ be a pair satisfying
  the Karush-Kuhn-Tucker conditions of~\eqref{pb:compact}:\footnote{We
    assume here that some constraint qualification condition holds,
    so that these conditions are necessary optimality conditions for
    Problem~\eqref{pb:compact}.}
  \begin{align*}
    & \nabla J(\mathbf{u}\opt) -\mu\opt \nabla \Theta(\mathbf{u}\opt)
      \in - N_{U\ad}(\mathbf{u}\opt) \eqsepv \\
    & \mu\opt \geq 0 ,\; \Theta(\mathbf{u}\opt) \geq p ,\;
      \mu\opt\big(\Theta(\mathbf{u}\opt)-p\big) = 0 \eqsepv
  \end{align*}
  where~$N_{U\ad}(\mathbf{u}\opt)$ is the normal cone to~$U\ad$
  at~$\mathbf{u}\opt$. Let~$\lambda\opt$ be defined by:
  \begin{equation*}
    \mu\opt = \frac{J(\mathbf{u}\opt)}{\Theta(\mathbf{u}\opt)} +
    \lambda\opt \Theta(\mathbf{u}\opt) \eqfinp
  \end{equation*}
  From the assumptions on~$\mu\opt$ and~$\Theta$, we deduce
  that~$\lambda\opt \geq 0$. Moreover, the condition
  $\mu\opt\big(\Theta(\mathbf{u}\opt)-p\big)=0$ implies:
  \begin{equation*}
    \frac{J(\mathbf{u}\opt)}{\Theta(\mathbf{u}\opt)}
    \big(\Theta(\mathbf{u}\opt)-p\big) +
    \lambda\opt \Theta(\mathbf{u}\opt)
    \big(\Theta(\mathbf{u}\opt)-p\big) =0 \eqsepv
  \end{equation*}
  and thus~$\lambda\opt\big(\Theta(\mathbf{u}\opt)-p\big)=0$ because
  both terms in the last expression are nonnegative. Finally,
  substituting~$\mu\opt$ by its expression function of~$\lambda\opt$
  in the first Karush-Kuhn-Tucker condition of Problem~\eqref{pb:compact}
  leads to:
  \begin{equation*}
    \nabla J(\mathbf{u}\opt) -
    \frac{J(\mathbf{u}\opt)}{\Theta(\mathbf{u}\opt)}\nabla\Theta(\mathbf{u}\opt) -
    \lambda\opt\Theta(\mathbf{u}\opt)\nabla\Theta(\mathbf{u}\opt)
    \in - N_{U\ad}(\mathbf{u}\opt) \eqfinp
  \end{equation*}
  We thus conclude that~$(\mathbf{u}\opt,\lambda\opt)$ satisfies
  the Karush-Kuhn-Tucker conditions of Problem~\eqref{pb:compacti}.
\end{proof}

\section{Choosing the smoothing parameter $r_k$}
\label{sec:analogy}

In~\cite{Andrieu_EJOR_2011} and~\cite{Andrieu_arXiv_2007}), a general problem of
minimizing a criterion in expectation under a probability constraint is studied.
This probability constraint, as in our study, is written as a constraint in
expectation involving the (discontinuous) Heaviside function~$Y$:
\begin{equation*}
  Y(y) =
  \begin{cases}
    0 & \text{if } y < 0 \eqfinv \\
    1 & \text{otherwise} \eqfinp
  \end{cases}
\end{equation*}
In order to numerically deal with this constraint in expectation, the
function~$Y$ is regularized by convolution (see~\cite{Ermoliev_SIAM_1995})
with a smoothing parameter~$r$. If the regularized function $Y_r$
is chosen symmetrical ($Y_r(- x )=Y_r(x)$), it is shown in~\cite[Theorem~3]{Andrieu_EJOR_2011} that minimizing the mean square
error of the estimate of the derivative of the smoothed constraint
leads to a parameter~$r$ proportional to~$k^{-1/5}$ when the estimation
is based on~$k$ samples. Choosing a non symmetrical regularized
function $Y_r$ such as
\begin{equation}
  \label{eq:Y-lissage}
  Y_{r}(y) =
  \begin{cases}
    0             & \text{if } y \leq -r \eqfinv \\
    1+\frac{y}{r} & \text{if } y \in ]- r,0[ \eqfinv \\
    1             & \text{otherwise} \eqfinv
  \end{cases}
\end{equation}
is possible and leads to choose a smoothing parameter~$r$
proportional to~$k^{-1/3}$.

% that the estimation of the derivative of the
% smoothed constraint has a bias in~$\mathcal{O}\bp{r^{2}}$ and a variance
% in~$\mathcal{O}\bp{1/r}$.  If the estimate is based on~$k$ samples, the variance is of
% order~$1/k.r$ and the square of the bias is of order~$r^{4}$. Therefore, the
% best trade-off between variance and bias is realized by that~$r$ which minimizes
% the mean square error (sum of the variance and the square of the bias), so that
% the best~$r$ is proportional to~$k^{-1/5}$, with a mean square error
% in~$k^{-4/5}$.
% In~\cite{Andrieu_EJOR_2011}, the smoothing is  performed symmetrically
% with respect to the origin. But, as shown in Appendix~\ref{sec:analogy},
% function~$\fctI_{r}$ in~\eqref{eq:smooth-indicator-function}
% corresponds to a non symmetrical smoothing of the indicator
% function~$\fctI$, so that the bias is in~$\mathcal{O}\bp{r}$
% instead of~$\mathcal{O}\bp{r^{2}}$ whereas the variance is
% unchanged in~$\mathcal{O}\bp{1/r}$. Then the best trade-off
% is obtained for~$r$ proportional to~$k^{-1/3}$, with a mean
% square error in~$k^{-2/3}$.

In our study, the function involved in the expected constraint is the indicator
function~$\fctI$ as defined in Equation~\eqref{eq:indicator-function}, instead
of the Heaviside function~$Y$ in~\cite{Andrieu_EJOR_2011}. As given
in Equation~\eqref{eq:smooth-indicator-function}, the indicator function~$\fctI$
is regularized as
\begin{equation}
  \label{eq:I-lissage}
  \fctI_{r}(y)= \max\Ba{0,1-\frac{y}{r}} \eqsepv y \geq 0 \eqfinp
\end{equation}
The regularized indicator function $\fctI_{r}$ enters the inner optimization
Problem~\eqref{eq:inner}.
%which is actually one of reaching the target within a tolerance~$r$,
%that is, an optimal control problem with a final constraint
%of the type~$\bnorm{C\bp{\xx(\tfx)}} \leq r$.
%
We show that this regularized problem involving function $\fctI_{r}$ in~\eqref{eq:I-lissage}
is equivalent to an other optimization problem involving function $Y_r$ in~\eqref{eq:Y-lissage}.
For that purpose, consider a failure $\xi=(\tpx,\tdx)$ and the satellite state $x(\tpx)$ at the
beginning of the failure. We define $G$ as a mapping which is non-negative when
reaching the target after the failure is possible and negative otherwise. Thus,
the space is divided into two parts according to the sign of the function~$G$,
the first part corresponding to the points from which the target can be reached.
Now, the condition~$C\bp{\xx(\tfx)}=0$ is equivalent to the
condition~$G\bp{x(\tpx),\tpx,\tdx} \geq 0$, that is,
\begin{equation}
  \label{eq:passage-IY}
  \fctI\Bp{\bnorm{C\bp{\xx(\tfx)}}} =
  Y\Bp{G\bp{x(\tpx),\tpx,\tdx}} \eqfinv
\end{equation}
and therefore, the analog of the smoothing function~\eqref{eq:I-lissage}
of~$\fctI$ for function~$Y$ is~\eqref{eq:Y-lissage}.  As already discussed, the
non symmetrical smoothing~$Y_r$ leads to a smoothing coefficient $r_k$
proportional to~$k^{-1/3}$.  Thus, we also choose a smoothing coefficient $r_k$
proportional to~$k^{-1/3}$ for $\fctI_{r}$.

\section{Remarks on the inner problem and its gradients}
\label{sec:gradient}

Here again, we use the function~$G$ introduced in Appendix~\ref{sec:analogy}.
Using~\eqref{eq:passage-IY}, the inner optimization problem~\eqref{eq:inner} becomes:
\begin{equation}
 \label{eq:inner-Y}
 W_{r}\bp{x(\tpx),\tpx,\tdx,\mu} =
 \min_{v}
 \bp{K\bp{\Xtf}-\mu} Y_{r}\Bp{G\bp{x(\tpx),\tpx,\tdx}}
 \eqfinp
\end{equation}
We are interested in the value and in the different partial derivatives of~$W_{r}$ (with respect the control~$v$, to the initial state~$x$ and
to the multiplier~$\mu$) as they appear in the stochastic Arrow-Hurwicz
algorithm in~\S\ref{ssect:stochastic-arrow}. There are several cases
to consider.
\begin{enumerate}
 \item $Y_{r}\bp{G\np{x(\tpx),\tpx,\tdx}} = 0$, that is,
 we miss the target by at least~$r$: then~$Y_{r}$ and its gradient
 are equal to zero, and so are~$W_{r}\bp{x(\tpx),\tpx,\tdx,\mu}$
 and its partial derivatives.
 \item $Y_{r}\bp{G\np{x(\tpx),\tpx,\tdx}} =1 + \theta/r$, that is,
 we approach the target at less than~$r$: then the partial derivatives
 are of order~$1/r$ since the derivative of~$Y_{r}$ is of order~$1/r$.
 \item $Y_{r}\bp{G\np{x(\tpx),\tpx,\tdx}} = 1$, that is,
 we hit the target exactly: then the gradient of~$Y_{r}$ is equal
 to zero, so that the partial derivatives of~$W_{r}$ only come
 from the final cost~$K$ and do not depend on~$r$.
\end{enumerate}

As a consequence, when one is at a distance less than~$r$ from
the target without being exactly there, the gradient of~$Y_{r}$
with respect to the control is of order~$1/r$, so that the inner
optimization problem will try of getting as close possible
of the target.

Another conclusion is that, in order to properly evaluate the gradients
of the function~$W_{r}$ with respect to the state~$x$ and the
multiplier~$\mu$ as required in the stochastic Arrow-Hurwicz algorithm
in~\S\ref{ssect:stochastic-arrow}, it is necessary to distinguish
the case when the target is hit exactly from the case where
the target is approached at less than~$r$.

\end{document}